\documentclass[11pt,a4paper,reqno]{amsart}

\usepackage[T1]{fontenc}
\usepackage[english]{babel}
\usepackage{amssymb}
\usepackage{libertinus}
\usepackage[cal=boondoxo,bb=ams]{mathalfa}
\usepackage{microtype}
\usepackage{enumitem}
\usepackage{etoolbox}
\usepackage{graphicx}
\usepackage{float}

\allowdisplaybreaks[2]
\usepackage[a4paper,left=30mm,right=30mm,top=27mm,bottom=30mm,headsep=8mm,footskip=13mm,heightrounded]{geometry}

\setlist[itemize]{leftmargin=2.1em,itemsep=0.25em,topsep=0.45em,parsep=0pt}
\setlist[enumerate]{leftmargin=2.1em,itemsep=0.25em,topsep=0.45em,parsep=0pt}

\numberwithin{equation}{section}

\makeatletter
\def\@settitle{%
 \begin{center}
  \vspace*{-0.4em}
  {\normalfont\bfseries\fontsize{17}{21}\selectfont\@title\par}
  \vspace{0.7em}
 \end{center}}
\renewcommand\section{\@startsection{section}{1}{\z@}%
 {1.5\baselineskip plus 0.2\baselineskip minus 0.1\baselineskip}%
 {0.65\baselineskip}{\normalfont\Large\bfseries}}
\renewcommand\subsection{\@startsection{subsection}{2}{\z@}%
 {1.15\baselineskip plus 0.2\baselineskip minus 0.1\baselineskip}%
 {0.45\baselineskip}{\normalfont\large\bfseries}}
\renewcommand\subsubsection{\@startsection{subsubsection}{3}{\z@}%
 {0.9\baselineskip plus 0.15\baselineskip minus 0.1\baselineskip}%
 {0.35\baselineskip}{\normalfont\normalsize\bfseries}}
\makeatother

\makeatletter
\patchcmd{\@setauthors}{\centering\footnotesize}{\centering\normalsize}{}{}
\patchcmd{\@setauthors}{\MakeUppercase{\authors}}{\authors}{}{}
\patchcmd{\maketitle}{\uppercasenonmath\shorttitle}{}{}{}
\patchcmd{\maketitle}{\@nx\MakeUppercase{\the\toks@}}{\the\toks@}{}{}
\makeatother

\usepackage{hyperref}
\hypersetup{hidelinks,pdfencoding=auto,
 pdftitle={Critical semilinear damped wave equations beyond moduli of continuity},
 pdfauthor={Wenhui Chen},
 pdfsubject={Semilinear damped wave equations with critical nonlinearities},
 pdfkeywords={damped wave equation, critical nonlinearity, Dini condition, Keller--Osserman criterion}}
\theoremstyle{plain}
\newtheorem{theorem}{Theorem}[section]
\newtheorem{lemma}[theorem]{Lemma}
\newtheorem{prop}[theorem]{Proposition}
\newtheorem{coro}[theorem]{Corollary}

\theoremstyle{definition}
\newtheorem{defn}[theorem]{Definition}

\newtheorem{assum}[theorem]{Assumption}

\newtheorem{example}[theorem]{Example}

\theoremstyle{remark}
\newtheorem{remark}[theorem]{Remark}

\newcommand{\ml}{\mathcal}
\newcommand{\mb}{\mathbb}
\newcommand{\dd}{\,\mathrm{d}}

\title[Critical damped waves beyond moduli]{Critical semilinear damped wave equations beyond\\ moduli of continuity}

\author[W. Chen]{Wenhui Chen}
\address{School of Mathematics and Information Science, Guangzhou University, Guangzhou 510006, P. R. China}
\email{wenhui.chen.math@gmail.com}

\keywords{semilinear damped wave equation, critical nonlinearity, Dini condition, global in-time existence, finite-time blow-up, global grow-up, Keller--Osserman criterion}
\subjclass[2020]{Primary 35L71, Secondary 35B33, 35B44, 35A01}
\date{}

\begin{document}

\begin{abstract}
  We consider the Cauchy problem for semilinear classical damped wave equations with general nonnegative nondecreasing nonlinearities, without imposing a modulus-of-continuity structure or convexity. A Dini-type integral condition expressed directly in terms of the nonlinearity gives small data global in-time existence and extends the nonlinear framework of Ebert--Girardi--Reissig [Math. Ann. 378 (2020), 1311--1326]. When this Dini-type integral diverges and the effective total mass of the initial data is positive, convergence of the Keller--Osserman tail integral is necessary and sufficient for finite-time blow-up of the maximal mild Sobolev solution, particularly, divergence of the tail integral gives global grow-up. The resulting integral criteria distinguish the contributions of the nonlinearity near zero and at infinity in governing small data global in-time existence with diffusion decay, global grow-up, and finite-time blow-up.
\end{abstract}

\maketitle

\section{Introduction}

We consider the following Cauchy problem for the semilinear classical damped wave equation with unknown $u=u(t,x)$:
\begin{align}\label{Main-Problem}
  \begin{cases}
    u_{tt}-\Delta u+u_t=F(|u|),&(t,x)\in(0,\infty)\times\mb R^n,\\
    (u,u_t)(0,x)=(u_0,u_1)(x),&x\in\mb R^n,
  \end{cases}
\end{align}
in space dimensions $n\in\{1,2\}$, where the general nonlinearity $F=F(s)$ is nonnegative and nondecreasing for $s\geqslant0$.
We seek a critical Dini-type condition directly for $F$ which gives small data global in-time existence without a modulus-of-continuity structure or convexity (to improve \cite{Ebert-Girardi-Reissig=2020,Dao-Reissig=2021,DAbbicco-Girardi=2023,Girardi=2025,Chen-Girardi=2025} essentially).
When the Dini-type integral diverges, we introduce the novel Keller--Osserman tail integral condition at infinity to distinguish finite-time blow-up from global in-time grow-up. To the best of our knowledge, this characterization in terms of the Keller--Osserman tail integral has not been established previously for semilinear damped wave equations.

For the power nonlinearity $F(s)=s^p$ with $p>1$, the threshold between small data global in-time existence and finite-time blow-up is the well-known Fujita exponent
\begin{align*}
  p_{\mathrm F}(n):=1+\frac{2}{n}\ \ \mbox{with}\ \ n\in\mb N,
\end{align*}
originally obtained for the semilinear heat equation in \cite{Fujita=1966}.
The corresponding critical results for the semilinear damped wave equation with $L^1$ integrable initial data was developed in \cite{Matsumura=1976,Li-Zhou=1995,Todorova-Yordanov=2001,Zhang=2001,Ikehata-Ohta=2002,Kirane-Qafsaoui=2002,Narazaki=2004,Ikehata-Tanizawa=2005} and references therein.
At the critical exponent $p=p_{\mathrm{F}}(n)$, the power (H\"older) scale was refined in \cite{Ebert-Girardi-Reissig=2020} by considering
\begin{align}\label{EGR-Nonlinearity}
  F(s)=s^{p_{\mathrm F}(n)}\mu(s)\ \ \mbox{for}\ \ s>0,\ \ \mbox{with}\ \ \mu(0)=0,
\end{align}
where the modulus of continuity $\mu$ is continuous, increasing and concave.
For $n\in\{1,2\}$, \cite[Theorem 3]{Ebert-Girardi-Reissig=2020} gives small data global in-time existence with diffusion decay under the Dini condition (see \cite{Dini=1880,Zygmund=1959})
\begin{align}\label{Dini-Mu}
  \int_0^{\delta_0}\frac{\mu(s)}{s}\dd s<\infty
  \ \ \mbox{for some}\ \ \delta_0>0,
\end{align}
together with the following differential bounds for some $s_0>0$:
\begin{align}\label{NDB}
  s^k|\mu^{(k)}(s)|\leqslant C\mu(s)
  \ \ \mbox{for}\ \ k\in\{1,\ldots,n\} \ \ \mbox{and} \ \ s\in(0,s_0].
\end{align}
For any $n\geqslant1$, \cite[Theorem 5]{Ebert-Girardi-Reissig=2020} treats the divergent Dini integral under convexity of the full nonlinearity
$z\mapsto |z|^{p_{\mathrm F}(n)}\mu(|z|)$ on $\mb R$.
It excludes global Sobolev solutions for initial data $(u_0,u_1)=(0,g)$ satisfying
\begin{align*}
  g\in H^{\left\lfloor\frac{n}{2}\right\rfloor}\cap L^1 \ \ \text{as well as}\ \  \int_{\mb R^n}g(x)\dd x>0.
\end{align*}
The critical regularity of nonlinearities has subsequently been studied for dissipative evolution equations in \cite{Dao-Reissig=2021,DAbbicco-Girardi=2023,Djaouti-Reissig=2023,Girardi=2025,Chen-Girardi=2025,Dao-Son=2025,Tang-Duong=2026,Tang-Dao-Cung=2026}, and for classical wave equations in \cite{Chen-Reissig=2024,Chen-Palmieri=2024,Wang-Zhang=2024,Shao=2024}.
We refer to \cite{Chen-Girardi-Review=2026} for an overview. 
Related questions on the critical behavior of nonlinearities have also been studied for semilinear shifted wave equations on hyperbolic spaces \cite{Wang-Zhang-Hyperbolic=2025,Zhang-Hyperbolic=2023}, where logarithmic corrections to $|u|$ near the origin were considered.

However, the arguments in \cite{Ebert-Girardi-Reissig=2020} (and some follow-up papers) rely on the structural assumptions on $\mu$ and the full nonlinearity in the following two ways.
\begin{itemize}
  \item The global in-time existence proof in \cite[Theorem 3]{Ebert-Girardi-Reissig=2020} uses the monotone increasing property of $\mu$ through
  \begin{align}\label{Freezing-Step}
    \mu\big(|u(t,x)|\big)\leqslant
    \mu\big(\|u(t,\cdot)\|_{L^\infty}\big),
  \end{align}
  so that the decay of the $L^\infty$ norm converts the nonlinear time integral into the Dini integral. 
  Related arguments also use monotonicity of the modulus to control the nonlinear source in \cite{Dao-Reissig=2021,DAbbicco-Girardi=2023,Djaouti-Reissig=2023,Chen-Reissig=2024,Girardi=2025,Chen-Girardi=2025,Tang-Duong=2026,Tang-Dao-Cung=2026}. 
  \item The nonexistence proof uses Jensen's inequality for the full convex nonlinearity and monotonicity of $\mu$ in the critical factor.
  Jensen's inequality also plays a central role in the blow-up arguments of subsequent works \cite{Dao-Reissig=2021,Chen-Reissig=2024,Chen-Girardi=2025,Chen-Palmieri=2024}, where convexity allows one to estimate nonlinear integral terms from below.
\end{itemize}
These results identify a sharp Dini threshold within the modulus-of-continuity framework under the above structural assumptions, but leave open the following question:
\begin{center}
  \emph{To what extent does Dini-type integrability determine global in-time solvability\\ and long-time behavior beyond this framework?}
\end{center}

In the present paper, we work directly with the nonlinearity $F$ and express its critical behavior near the origin through the Dini-type integral
\begin{align}\label{Critical-Integral}
  \ml I_F(\delta):=\int_0^\delta\frac{F(s)}{s^{p_{\mathrm F}(n)+1}}\dd s
  \ \  \mbox{with}\ \ \delta>0.
\end{align}
For \eqref{EGR-Nonlinearity}, this is exactly the Dini integral in \eqref{Dini-Mu}.
Under Assumption~\ref{Assum-Chen}, its convergence yields small data global in-time existence with diffusion decay by Theorem~\ref{Thm-Main}, extending the role of Dini-type integrability beyond the modulus-of-continuity framework.
To overcome the difficulty caused by the lack of monotonicity of $\mu$, we use amplitude-layer estimates to bound the nonlinear source through a layer-cake representation with respect to $\mathrm dF$, while Hardy--Stieltjes estimates control the nonlinear time convolutions without differential bounds (for example, \eqref{NDB}) on the critical quotient $\frac{F(s)}{s^{p_{\mathrm F}(n)}}$. Moreover, the family in Example~\ref{Example-Beyond} shows that neither nondecreasing majorants of the critical quotient nor convex minorants of $F$ recover, in general, the threshold determined by $\ml I_F$.

When the critical integral \eqref{Critical-Integral} diverges, every global weak solution with data of positive effective total mass \eqref{Effective-Mass} as in Theorem~\ref{Thm-Blowup} must become persistently large on expanding balls $\{|x|\leqslant\sqrt t\}$ at the diffusion scale.
To determine whether this growth can continue for all time or leads to finite-time blow-up, we introduce the Keller--Osserman tail integral (see \cite{Keller=1957,Osserman=1957})
\begin{align}\label{Tail-Integral}
  \mathcal K_\infty(F):=\int_1^\infty
  \frac{\dd s}{\sqrt{1+\int_0^sF(r)\dd r}}.
\end{align}
Under Assumption~\ref{Assum-Chen}, Theorem~\ref{Thm-Tail-Global} gives global in-time existence for arbitrary Sobolev data whenever $\mathcal K_\infty(F)=\infty$, independently of the near-zero integral.
Together with the persistent growth forced by $\ml I_F(\delta)=\infty$, this yields global grow-up for data with positive effective total mass \eqref{Effective-Mass}.
If instead $\mathcal K_\infty(F)<\infty$, the same divergent near-zero condition and positive effective mass lead to finite-time blow-up that excludes global weak solutions.
Thus, under these conditions, the Keller--Osserman tail integral distinguishes finite-time blow-up from global grow-up, as stated in Theorem~\ref{Thm-Blowup}.
For related integral criteria in semilinear heat-type equations, one may see \cite{Kobayashi-Sirao-Tanaka=1977,Laister-Sierzega=2021,Chung-Hwang=2024,Kumar-Torebek=2025} and references therein.

To prove the growth and blow-up assertions for possibly sign-changing data, we combine positivity of the velocity kernel with the diffusion asymptotic of the full linear solution to obtain a positive linear lower bound.
A delayed comparison based on a running minimum then yields persistent growth on expanding balls.
When the Keller--Osserman tail integral converges, this growth allows comparison from below with a solution of the damped scalar equation on a spatial cone, which gives finite-time blow-up.
The use of kernel positivity and nonlinear lower bounds in Section~\ref{Sec-Growth} follows the approach developed recently for semilinear strongly damped wave equations in \cite{Chen-Strong-3D=2026,Chen-Strong-2D=2026}.
\medskip

\paragraph{Notation.}
The constants $C$ and $c$ are positive and may change from line to line.
We write $f\lesssim g$ if $f\leqslant Cg$, and $f\approx g$ if both $f\lesssim g$ and $g\lesssim f$ hold.
We write $f(s)\sim g(s)$ if $\frac{f(s)}{g(s)}\to1$ in the indicated limit.
We use $\langle t\rangle:=1+t$ for the time variable. Lebesgue and Sobolev spaces, and their norms, are taken over $\mb R^n$ unless a domain is displayed. We denote by $|E|$ the Lebesgue measure of a measurable set $E$. For a nondecreasing function $F$, $\mathrm dF$ denotes its Lebesgue--Stieltjes measure.
Under Assumption~\ref{Assum-Chen}, the Stieltjes integrals may equivalently be read as integrals against $F'(s)\dd s$. The Fourier transform and its inverse are normalized by
\begin{align*}
  \widehat h(\xi)&:=\int_{\mb R^n}\mathrm e^{-\mathrm ix\cdot\xi}h(x)\dd x
  \ \ \text{and}\ \  h(x):=\frac{1}{(2\pi)^n}\int_{\mb R^n}
  \mathrm e^{\mathrm ix\cdot\xi}\widehat h(\xi)\dd\xi.
\end{align*}
We denote by $B_R:=\{x\in\mb R^n:|x|<R\}$ the open ball centered at the origin.

\section{Main results}\label{Sec-Main-Results}

We consider the following data space:
\begin{align*}
  \ml A_n:=\big(H^{N_n}\cap L^1\big)\times
  \big(H^{N_n-1}\cap L^1\big)
  \ \  \mbox{with}\ \ N_n:=1+\left\lfloor\frac{n}{2}\right\rfloor,
\end{align*}
endowed with the corresponding natural product norm, as in \cite{Ebert-Girardi-Reissig=2020}.

\begin{assum}[Nonlinearity]\label{Assum-Chen}
  The function $F:[0,\infty)\to[0,\infty)$ is nondecreasing, $F(0)=0$, and its even extension
  $\ml F_e(z):=F(|z|)$ for $z\in\mb R$ belongs to $\ml C^{N_n}$.
\end{assum}
\begin{remark}
  For nonlinearities of the form $F(s)=s^{p_{\mathrm F}(n)}\mu(s)$, the local differential assumptions \eqref{NDB} stated in \cite{Ebert-Girardi-Reissig=2020} already ensure the required regularity of the even extension at zero.
  The condition $\ml F_e\in\ml C^{N_n}$ concerns the full nonlinearity and does not require its critical quotient $\frac{F(s)}{s^{p_{\mathrm F}(n)}}$ to be monotone or to satisfy scale-invariant differential bounds.
\end{remark}

\subsection{Small data global in-time existence under the Dini-type condition}

\begin{defn}[Mild Sobolev solution]\label{Defn-Mild-Sobolev-Solution}
  For $0<T\leqslant\infty$, a function
  \begin{align*}
    u\in\ml C\big([0,T),H^{N_n}\big)
    \cap\ml C^1\big([0,T),H^{N_n-1}\big)
  \end{align*}
  is a mild Sobolev solution to the Cauchy problem \eqref{Main-Problem} if
  \begin{align*}
    u(t,\cdot)=K_0(t)u_0(\cdot)+K_1(t)u_1(\cdot)
    +\int_0^tK_1(t-s)F\big(|u(s,\cdot)|\big)\dd s
    \ \  \mbox{in}\ \ H^{N_n}
  \end{align*}
  for every $t\in[0,T)$.
  Here, $K_0(t)$ and $K_1(t)$ are the position and velocity propagators, respectively, for the homogeneous classical damped wave equation, carrying $K_0(t)=\partial_tK_1(t)+K_1(t)$. The kernel representation and properties of $K_1(t)$ are given in Lemma~\ref{lem:kernel}.
\end{defn}

\begin{theorem}[Small data global in-time existence]\label{Thm-Main}
  Let $F$ satisfy Assumption~\ref{Assum-Chen} and
  \begin{align}\label{Intrinsic-Dini}
    \ml I_F(\delta_0)<\infty
    \ \  \mbox{for some}\ \ \delta_0>0.
  \end{align}
  There exists $\varepsilon_0=\varepsilon_0(n,F,\delta_0)>0$ such that, for every $(u_0,u_1)\in\ml A_n$ with $\|(u_0,u_1)\|_{\ml A_n}\leqslant\varepsilon\leqslant\varepsilon_0$,
  the Cauchy problem \eqref{Main-Problem} admits a unique global in-time mild Sobolev solution in the sense of Definition~\ref{Defn-Mild-Sobolev-Solution}.
  Furthermore, with $m=2$ if $n=1$ and with $\frac{10}{9}\leqslant m<2$ if $n=2$, the solution satisfies the next diffusion decay estimates:
  \begin{align*}
    \|u(t,\cdot)\|_{L^m}&\lesssim\varepsilon\langle t\rangle^{-\frac{n}{2}\left(1-\frac{1}{m}\right)},\\
    \|\nabla^ju(t,\cdot)\|_{L^2}&\lesssim\varepsilon\langle t\rangle^{-\frac{n}{4}-\frac{j}{2}}
    \ \  \mbox{for}\ \ j\in\{0,1\},\\
    \|u(t,\cdot)\|_{L^\infty} &\lesssim\varepsilon\langle t\rangle^{-\frac{n}{2}},
  \end{align*}
  for every $t\geqslant0$.
\end{theorem}

\begin{remark}
  For the modulus-type source \eqref{EGR-Nonlinearity}, Theorem~\ref{Thm-Main} gives the corresponding $L^2$-based and $L^\infty$ diffusion decay in the same data space under the Dini condition \eqref{Dini-Mu}.
  The decay rates in Theorem~\ref{Thm-Main} coincide with those for the corresponding linearized Cauchy problem in the stated norms \cite{Matsumura=1976,Ikeda-Inui-Okamoto-Wakasugi=2019}, so there is no loss of decay.
\end{remark}

\begin{remark}
  Concerning $n=2$, the restriction $m\geqslant\frac{10}{9}$ arises from controlling the high-frequency derivative loss in the estimate for $K_0(t)u_0$ under the assumption $u_0\in H^2\cap L^1$, as shown in Lemma~\ref{Lemma-Position-Lm}.
  If $u_0=0$, the stated $L^m$ decay estimate holds for every $1<m<2$.
\end{remark}

\subsection{The Keller--Osserman classification}

The tail integral \eqref{Tail-Integral} gives a global in-time existence criterion in the next result for arbitrary data in $H^{N_n}\times H^{N_n-1}$, independently of the convergence or divergence of $\ml I_F(\delta_0)$.
Compared with Theorem~\ref{Thm-Main}, neither smallness nor additional $L^1$ integrability of the initial data is required, but the resulting global in-time solutions need not satisfy diffusion decay estimates and may exhibit a global grow-up phenomenon.

\begin{theorem}[Global in-time existence for arbitrary Sobolev data]\label{Thm-Tail-Global}
  Let $F$ satisfy Assumption~\ref{Assum-Chen} and
  \begin{align}\label{KO-Infty}
    \mathcal K_\infty(F)=\infty.
  \end{align}
  For every $(u_0,u_1)\in H^{N_n}\times H^{N_n-1}$, the Cauchy problem \eqref{Main-Problem} admits a unique global in-time mild Sobolev solution in the sense of Definition~\ref{Defn-Mild-Sobolev-Solution}.
\end{theorem}

\begin{remark}\label{Rem-KO-Modulus}
  The condition \eqref{KO-Infty} excludes the power nonlinearities $F(s)=s^p$ with any $p>1$, for which $\mathcal K_\infty(F)<\infty$.
  The same fact holds for \eqref{EGR-Nonlinearity} whenever $\mu$ is nontrivial and nondecreasing on $[0,\infty)$, as in the modulus-of-continuity framework of \cite{Ebert-Girardi-Reissig=2020}.
  Typical examples satisfying Assumption~\ref{Assum-Chen} and \eqref{KO-Infty} are
  \begin{align*}
    F(s)=\frac{s^2}{1+s},\ \ F(s)=s\log(1+s),\ \ \mbox{and}\ \ F(s)=s[\log(1+s)]^2.
  \end{align*}
\end{remark}

Before stating the nonexistence result below, let us introduce the following notion of weak solution, which includes the mild Sobolev solutions from Definition~\ref{Defn-Mild-Sobolev-Solution} with initial data in $\ml A_n$.

\begin{defn}[Weak solution]\label{Defn-Weak}
  Let $0<T\leqslant\infty$ and $(u_0,u_1)\in\ml A_n$.
  A function $u=u(t,x)$ is called a weak solution to \eqref{Main-Problem} on $[0,T)$ if
  $u,F(|u|)\in L^1_{\mathrm{loc}}\big([0,T)\times\mb R^n\big)$ and
  \begin{align}\label{Weak-Identity}
    &\int_0^T\int_{\mb R^n}F\big(|u(t,x)|\big)\phi(t,x)\dd x\dd t
    +\int_{\mb R^n}\big[\big(u_0(x)+u_1(x)\big)\phi(0,x)-u_0(x)\phi_t(0,x)\big]\dd x\notag\\
    &=\int_0^T\int_{\mb R^n}u(t,x)
    \big(\phi_{tt}(t,x)-\Delta\phi(t,x)-\phi_t(t,x)\big)\dd x\dd t
  \end{align}
  for every $\phi\in\ml C_c^2\big([0,T)\times\mb R^n\big)$.
\end{defn}

\begin{theorem}[Finite-time blow-up and global grow-up]\label{Thm-Blowup}
  Let $F$ satisfy Assumption~\ref{Assum-Chen} and
  \begin{align}\label{Intrinsic-NonDini}
    \ml I_F(\delta_0)=\infty
    \ \ \mbox{for some}\ \ \delta_0>0.
  \end{align}
  Suppose that $(u_0,u_1)\in\ml A_n$ has positive effective mass, namely,
  \begin{align}\label{Effective-Mass}
    M_0:=\int_{\mb R^n}\big(u_0(x)+u_1(x)\big)\dd x>0.
  \end{align}
  Then the maximal existence time $T_{\max}$ of the mild Sobolev solution satisfies
  \begin{align}\label{Tail-Dichotomy-Equivalence}
    T_{\max}<\infty\ \ \Leftrightarrow\ \ \mathcal K_\infty(F)<\infty.
  \end{align}
  More precisely, the following alternatives hold.
  \begin{enumerate}[label=\textup{(\roman*)}]
    \item If $\mathcal K_\infty(F)<\infty$, there is no global weak solution, and
    \begin{align}\label{Finite-Blowup}
      \limsup_{t\to T_{\max}^-}\|u(t,\cdot)\|_{L^\infty}=\infty.
    \end{align}
    \item If $\mathcal K_\infty(F)=\infty$, the unique mild Sobolev solution is global in-time and its continuous representative satisfies
    \begin{align}\label{Global-Growup}
      \lim_{t\to\infty}\inf_{|x|\leqslant\sqrt t}u(t,x)=+\infty,
    \end{align}
    where $|x|\leqslant\sqrt t$ describes an expanding ball on the diffusion scale of the damped wave equation.
  \end{enumerate}
\end{theorem}

\begin{remark}\label{Rem-Weak-Growup}
  Under the near-zero and data assumptions of Theorem~\ref{Thm-Blowup}, Corollary~\ref{cor:growup} gives, for every global weak solution and each $A>0$, that
  \begin{align}\label{Weak-Growup}
    u(t,x)\geqslant A
    \ \  \mbox{for almost every }(t,x)
    \mbox{ with }t\geqslant T_A \ \ \mbox{and} \ \  |x|\leqslant\sqrt t,
  \end{align}
  for some $T_A>0$.
  For the global mild Sobolev solution in Theorem~\ref{Thm-Blowup} \textup{(ii)}, continuity yields the pointwise lower bound in \eqref{Weak-Growup} and hence \eqref{Global-Growup}.
\end{remark}

\begin{remark}
  Theorem~\ref{Thm-Blowup} gives persistent growth on expanding balls through \eqref{Global-Growup}, but does not determine its rate precisely.
  A natural and open question is whether, under the same assumptions in the global grow-up case, one can obtain matching upper and lower bounds in $t$ for the quantity $\inf_{|x|\leqslant\sqrt t}u(t,x)$ in terms of the large-amplitude behavior of $F$.
  A main difficulty is that the scalar lower comparison on spatial cones is restricted to finite time intervals and does not directly yield a lower bound matching the global upper comparison, particularly without additional structural conditions on $F$ at infinity.
\end{remark}

\subsection{\texorpdfstring{Near-zero and large-amplitude behavior of the nonlinearity}{Near-zero and large-amplitude behavior of the nonlinearity}}\label{Sec-Example}

The integrals \eqref{Critical-Integral} and \eqref{Tail-Integral} impose conditions on different amplitude ranges of $F$.
\begin{itemize}
  \item The convergence of $\ml I_F(\delta)$ depends only on the behavior of $F$ near zero and measures the integrability of the critical quotient $\frac{F(s)}{s^{p_{\mathrm F}(n)}}$ with respect to $\frac{\dd s}{s}$.
  Under Assumption~\ref{Assum-Chen}, its convergence implies
  \begin{align*}
    F(s)=o\big(s^{p_{\mathrm F}(n)}\big)\ \ \mbox{as}\ \ s\to0^+
  \end{align*}  
  by \eqref{F-Little-o}, but this pointwise smallness alone does not ensure convergence, as Example~\ref{Example-Beyond} shows later.
  \item The convergence of $\mathcal K_\infty(F)$ depends on the growth of the primitive of $F$ at infinity.
  By Lemma~\ref{lem:tail}, its convergence implies
  \begin{align*}
    \frac{F(s)}{s}\to\infty\ \ \mbox{as}\ \ s\to\infty,
  \end{align*}
  so $F$ must grow faster than linearly, although this asymptotic property alone does not ensure convergence.
  Its convergence is unaffected by changes to $F$ on bounded intervals, which alter its primitive only by an additive constant for large $s$.
\end{itemize}

Under Assumption~\ref{Assum-Chen}, the near-zero behavior and the tail can be prescribed independently by joining nonnegative derivatives, as in Appendix~\ref{App-Tail-Join}.

\begin{example}[Modulus-of-continuity nonlinearities]\label{Example-MoC}
  For $F$ of the form \eqref{EGR-Nonlinearity} satisfying Assumption~\ref{Assum-Chen}, the integral \eqref{Critical-Integral} reduces to \eqref{Dini-Mu}.
  If $\mu$ is a nontrivial modulus of continuity on $[0,\infty)$, then $\mathcal K_\infty(F)<\infty$ by Remark~\ref{Rem-KO-Modulus}, which means that Theorems~\ref{Thm-Main} and \ref{Thm-Blowup} recover the Dini threshold in \cite{Ebert-Girardi-Reissig=2020}.
\end{example}

\begin{example}[A staircase family beyond moduli of continuity]\label{Example-Beyond}
  We construct a family for which neither nondecreasing majorants of the critical quotient $\frac{F(s)}{s^{p_{\mathrm F}(n)}}$ nor convex minorants of $F$ preserve the critical-integral criterion.

  Fix a constant $a>0$. Starting from $a_1:=0$, we define
  \begin{align}\label{Staircase-Cycles}
    \begin{cases}
      \displaystyle{b_j:=a_j+j^2},\\[0.3em]
      \displaystyle{a_{j+1}:=b_j+(j+1)^2+\frac{a}{p_{\mathrm F}(n)}\log\frac{j+1}{j},}
    \end{cases}
  \end{align}
  for $j\in\mb N$.
  The logarithmic prototype is
  \begin{align}\label{Staircase-Prototype}
    f_{0,a}(t):=
    \begin{cases}
      \mathrm e^{-p_{\mathrm F}(n)-2p_{\mathrm F}(n)t}&\mbox{if}\ \ t\leqslant0,\\
      j^{-a}\mathrm e^{-p_{\mathrm F}(n)b_j}&\mbox{if}\ \ a_j\leqslant t\leqslant b_j,\\
      j^{-a}\mathrm e^{-p_{\mathrm F}(n)b_j-2p_{\mathrm F}(n)(t-b_j)}
      &\mbox{if}\ \ b_j\leqslant t\leqslant a_{j+1}.
    \end{cases}
  \end{align}
  Under the change of variables $s=\mathrm e^{-t}$, the branch $t\leqslant0$ prescribes a power-law tail for large $s$.
  For $t\geqslant0$, the constant and exponentially decaying pieces keep $f_{0,a}$ nonincreasing, while $\mathrm e^{p_{\mathrm F}(n)t}f_{0,a}(t)$ alternately increases and decreases, as illustrated in Figure~\ref{Fig-Staircase} for $n=2$ and $a=2$.
  Its peak at $t=b_j$ has height $j^{-a}$, and its integral over $[a_j,a_{j+1}]$ is comparable to $j^{-a}$, so convergence of the critical integral is determined by the series $\sum_{j=1}^\infty j^{-a}$.

  \begin{figure}[H]
    \centering
    \includegraphics[width=\textwidth]{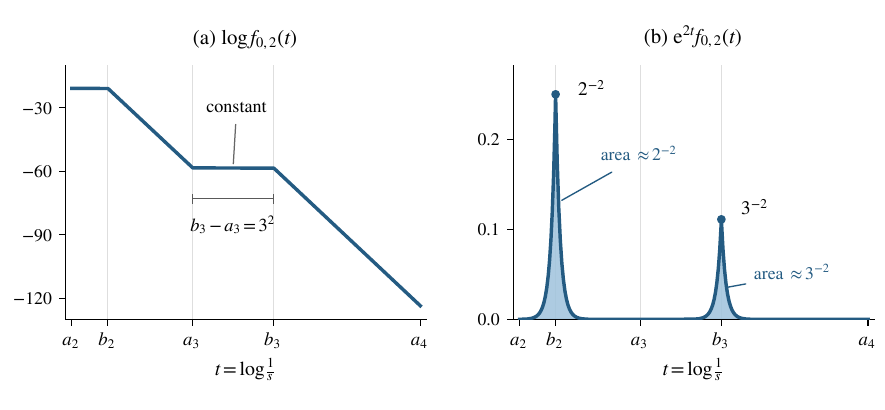}
    \caption{The prototype in \eqref{Staircase-Prototype} for $n=2$ and $a=2$.}
    \label{Fig-Staircase}
  \end{figure}

  Let us choose an even nonnegative function $\rho\in\ml C_c^\infty\big((-\frac{1}{4},\frac{1}{4})\big)$ such that $\int_{\mb R}\rho(\tau)\dd\tau=1$.
  We smooth the logarithm of the prototype $f_{0,a}$ by
  \begin{align}\label{Staircase-Smoothing}
    \ell_a(t):=\int_{-\frac{1}{4}}^{\frac{1}{4}}
    \rho(\tau)\log f_{0,a}(t-\tau)\dd\tau.
  \end{align}
  Returning to the variable $s$, we take the nonlinearity $F_a$ for this example via
  \begin{align}\label{Staircase-Nonlinearity}
    F_a(s):=\exp\left(\ell_a\left(\log\frac{1}{s}\right)\right)
    \ \ \text{for}\ \ s>0,
    \ \ \text{with}\ \  F_a(0):=0.
  \end{align}
  By Proposition~\ref{Prop-Obstruction}, $F_a$ satisfies Assumption~\ref{Assum-Chen} and its critical quotient $\frac{F_a(s)}{s^{p_{\mathrm F}(n)}}$ is nonmonotone on every neighborhood of zero. The calculation in the proof of Proposition~\ref{Prop-Obstruction} gives
  \begin{align}\label{Staircase-Threshold}
    \ml I_{F_a}(1)\approx\sum_{j=1}^\infty j^{-a}<\infty\ \ \Leftrightarrow\ \ a>1.
  \end{align}

  The following two properties show why comparison with nonlinearities satisfying additional monotonicity or convexity conditions may fail to preserve the critical integral criterion.
  \begin{itemize}
    \item If $1<a\leqslant3$, then $\ml I_{F_a}(1)<\infty$ by \eqref{Staircase-Threshold}.
    However, every nonnegative nondecreasing function $\nu$ satisfying
    \begin{align*}
      F_a(s)\leqslant Cs^{p_{\mathrm F}(n)}\nu(s)
      \ \ \mbox{for}\ \ 0<s<\delta
    \end{align*}
    for some $C,\delta>0$ necessarily satisfies $\int_0^\delta\frac{\nu(s)}{s}\dd s=\infty$. Namely, no such $\nu$ satisfies the Dini condition \eqref{Dini-Mu}, although Theorem~\ref{Thm-Main} applies directly to $F_a$.
    \item If $0<a\leqslant1$, then $\ml I_{F_a}(1)=\infty$ by \eqref{Staircase-Threshold}.
    However, every nonnegative convex function $G$ with $G(0)=0$ satisfying
    \begin{align*}
      G(s)\leqslant F_a(s)
      \ \ \mbox{for}\ \ 0<s<\delta
    \end{align*}
    for some $\delta>0$ necessarily satisfies $\ml I_G(\delta)<\infty$.
    Similarly, every nonnegative nondecreasing function $\nu$ satisfying
    \begin{align*}
      s^{p_{\mathrm F}(n)}\nu(s)\leqslant F_a(s)
      \ \ \mbox{for}\ \ 0<s<\delta
    \end{align*}
    necessarily satisfies $\int_0^\delta\frac{\nu(s)}{s}\dd s<\infty$. Namely, neither type of lower bound preserves the divergence required in \eqref{Intrinsic-NonDini}.
  \end{itemize}

  Since
  \begin{align*}
    F_a(s)=\mathrm e^{-p_{\mathrm F}(n)}s^{2p_{\mathrm F}(n)}
    \ \ \mbox{for}\ \ s\geqslant\mathrm e^{\frac{1}{4}},
  \end{align*}
  we have $\mathcal K_\infty(F_a)<\infty$ for every $a>0$.
  Together with \eqref{Staircase-Threshold}, Theorem~\ref{Thm-Main} gives small data global in-time existence with diffusion decay for $a>1$, whereas Theorem~\ref{Thm-Blowup} gives finite-time blow-up for $0<a\leqslant1$ and data in $\ml A_n$ with positive effective mass \eqref{Effective-Mass}.
\end{example}

\begin{remark}\label{Rem-Tails}
  For $a>0$ and the additional index $\gamma\geqslant0$, the construction in Appendix~\ref{App-Tail-Join} gives a nonlinearity $F_{a,\gamma}$ satisfying Assumption~\ref{Assum-Chen} with
  \begin{align*}
    F_{a,\gamma}(s)&=F_a(s)\ \ \ \ \ \ \ \mbox{for}\ \ 0\leqslant s\leqslant2,\\
    F_{a,\gamma}(s)&\sim s(\log s)^\gamma\ \ \mbox{as}\ \ s\to\infty.
  \end{align*}
  The criterion \eqref{Staircase-Threshold} and both comparison properties in Example~\ref{Example-Beyond} remain valid for $F_{a,\gamma}$, while
  \begin{align}\label{Tail-Logarithmic-Threshold}
    \mathcal K_\infty(F_{a,\gamma})<\infty
    \ \ \Leftrightarrow\ \ \gamma>2.
  \end{align}
  For $0<a\leqslant1$ and data in $\ml A_n$ with positive effective mass \eqref{Effective-Mass}, Theorem~\ref{Thm-Blowup} gives global grow-up if $0\leqslant \gamma\leqslant2$ and finite-time blow-up if $\gamma>2$.
\end{remark}

\section{Small data global in-time existence and decay}\label{Sec-Small}

\subsection{Local theory and linear estimates}

\begin{lemma}[Local Sobolev theory and continuation]\label{Lemma-Local}
  Let $n\in\{1,2\}$ and let $F$ satisfy Assumption~\ref{Assum-Chen}. For every $(u_0,u_1)\in H^{N_n}\times H^{N_n-1}$, there exists a unique maximal mild Sobolev solution
  \begin{align*}
    u\in\ml C\bigl([0,T_{\max}),H^{N_n}\bigr)\cap\ml C^1\bigl([0,T_{\max}),H^{N_n-1}\bigr).
  \end{align*}
  The maximal existence time satisfies the continuation criterion
  \begin{align}\label{Continuation-Criterion}
    \sup_{0\leqslant t<T_{\max}}\|u(t,\cdot)\|_{L^\infty}<\infty\ \ \Rightarrow\ \ T_{\max}=\infty.
  \end{align}
\end{lemma}

\begin{proof}
  For $n=1$, the embedding $H^1\hookrightarrow L^\infty$ and $\ml F_e\in\ml C^1$ give local Lipschitz continuity of $u\mapsto\ml F_e(u)$ from $H^1$ to $L^2$. For $n=2$ and $u,v$ in a bounded subset of $H^2$, we obtain
  \begin{align*}
    \nabla\big(\ml F_e(u)-\ml F_e(v)\big)(x)=\ml F_e'(u(x))\nabla(u-v)(x)+\big(\ml F_e'(u(x))-\ml F_e'(v(x))\big)\nabla v(x).
  \end{align*}
  For $\|u\|_{H^2}+\|v\|_{H^2}\leqslant R$, the embedding $H^2\hookrightarrow L^\infty$ bounds the ranges of $u$ and $v$ in a compact interval depending only on $R$.
  The boundedness of $\ml F_e''$ on this interval and the mean value theorem yield
  \begin{align*}
    \|\ml F_e'(u)-\ml F_e'(v)\|_{L^\infty}
    \leqslant C_R\|u-v\|_{L^\infty}.
  \end{align*}
  Using the preceding gradient identity, H\"older's inequality and the boundedness of $\ml F_e'$ on the same interval, we obtain
  \begin{align*}
    \|\ml F_e(u)-\ml F_e(v)\|_{H^1}
    &\leqslant C_R\big(\|u-v\|_{H^1}
    +\|u-v\|_{L^\infty}\|\nabla v\|_{L^2}\big)
    \leqslant C_R\|u-v\|_{H^2}.
  \end{align*}
  The linear propagator formula and these difference estimates yield a unique local in-time mild Sobolev solution by the standard fixed-point argument in \cite[Chapter~6, Theorem~1.4]{Pazy=1983}, which extends uniquely to a maximal interval $[0,T_{\max})$.

  Assume that
  \begin{align*}
    \sup_{0\leqslant t<T_{\max}}\|u(t,\cdot)\|_{L^\infty}\leqslant R.
  \end{align*}
  Because of $\ml F_e(0)=0$, the mean value theorem and the chain rule yield
  \begin{align}\label{Local-Composition-Estimate}
    \bigl\|\ml F_e\bigl(u(t,\cdot)\bigr)\bigr\|_{H^{N_n-1}}
    &=\bigl\|\ml F_e\bigl(u(t,\cdot)\bigr)-\ml F_e(0)\bigr\|_{H^{N_n-1}}\notag\\
    &\leqslant\sup_{|s|\leqslant R}|\ml F_e'(s)|\,\|u(t,\cdot)\|_{H^{N_n-1}}\notag\\
    &\lesssim_R\|u(t,\cdot)\|_{H^{N_n-1}}.
  \end{align}
  To control $(u,u_t)$ in $H^{N_n}\times H^{N_n-1}$, we introduce the following Sobolev energy:
  \begin{align*}
    E_{N_n}(t):=\frac{1}{2}\sum_{|\alpha|\leqslant N_n-1}
    \left(\|D^\alpha u_t(t,\cdot)\|_{L^2}^2
    +\|\nabla D^\alpha u(t,\cdot)\|_{L^2}^2
    +\|D^\alpha u(t,\cdot)\|_{L^2}^2\right),
  \end{align*}
  where $D^\alpha:=\partial_{x_1}^{\alpha_1}\cdots\partial_{x_n}^{\alpha_n}$ for $\alpha\in\mb N_0^n$.
  The energy identity, justified by spatial regularization, together with Young's inequality and \eqref{Local-Composition-Estimate}, gives
  \begin{align*}
    E_{N_n}'(t)
    &=-\sum_{|\alpha|\leqslant N_n-1}\|D^\alpha u_t(t,\cdot)\|_{L^2}^2
    +\sum_{|\alpha|\leqslant N_n-1}
    \left\langle D^\alpha\ml F_e(u(t,\cdot))+D^\alpha u(t,\cdot),
    D^\alpha u_t(t,\cdot)\right\rangle_{L^2}\\
    &\leqslant-\frac{1}{2}\sum_{|\alpha|\leqslant N_n-1}
    \|D^\alpha u_t(t,\cdot)\|_{L^2}^2
    +C\left(\|\ml F_e(u(t,\cdot))\|_{H^{N_n-1}}^2
    +\|u(t,\cdot)\|_{H^{N_n-1}}^2\right)\\
    &\leqslant C_R E_{N_n}(t).
  \end{align*}
  Here, $C_R$ depends on $n$ and the bound for $\ml F_e'$ on $[-R,R]$.
  Gronwall's inequality yields $E_{N_n}(t)\leqslant E_{N_n}(0)\mathrm e^{C_Rt}$ and hence, if $T_{\max}<\infty$,
  \begin{align*}
    \sup_{0\leqslant t<T_{\max}}\|(u,u_t)(t,\cdot)\|_{H^{N_n}\times H^{N_n-1}}\lesssim_{R,T_{\max}}\|(u_0,u_1)\|_{H^{N_n}\times H^{N_n-1}}.
  \end{align*}
  The local existence time can therefore be chosen uniformly when the problem is restarted at $t_0$ sufficiently close to $T_{\max}$. The resulting solution extends beyond $T_{\max}$, contradicting maximality and proving \eqref{Continuation-Criterion}.
\end{proof}

We set $\delta_m:=\frac{n}{2}(1-\frac{1}{m})$.
We use the following velocity-propagator estimates from \cite{Matsumura=1976,Marcati-Nishihara=2003,Hosono-Ogawa=2004,Narazaki=2004,Ikeda-Inui-Okamoto-Wakasugi=2019,DAbbicco-Ebert=2025}. For $n=1$, we take $m=r=2$. For $n=2$, we fix $1<m<2<r<\infty$.

\begin{lemma}[Estimates for $K_1(t)$]
  Let $n\in\{1,2\}$ and $j\in\{0,1\}$. There exists $c>0$ such that, for every $t\geqslant0$,
  \begin{align}
    \|K_1(t)f(\cdot)\|_{L^m}&\lesssim\langle t\rangle^{-\delta_m}\|f\|_{L^1},\label{Linear-L1-Lm}\\
    \|\nabla^jK_1(t)f(\cdot)\|_{L^2}&\lesssim\langle t\rangle^{-\frac{n}{4}-\frac{j}{2}}\|f\|_{L^1}+\mathrm e^{-ct}\|f\|_{L^2},\label{Linear-Energy-Mixed}\\
    \|\nabla^jK_1(t)f(\cdot)\|_{L^2}&\lesssim\langle t\rangle^{-\frac{j}{2}}\|f\|_{L^2},\label{Linear-L2-L2}\\
    \|K_1(t)f(\cdot)\|_{L^\infty}&\lesssim\langle t\rangle^{-\frac{n}{2}}\|f\|_{L^1}+\mathrm e^{-ct}\|f\|_{L^r},\label{Linear-Linfty-Mixed}\\
    \|K_1(t)f(\cdot)\|_{L^\infty}&\lesssim\langle t\rangle^{-\frac{n}{2r}}\|f\|_{L^r}.\label{Linear-Lr-Linfty}
  \end{align}
\end{lemma}
The position propagator requires a separate estimate in dimension two.

\begin{lemma}[$L^m$ estimate for $K_0(t)$]\label{Lemma-Position-Lm}
  Let $n=2$, $\frac{10}{9}\leqslant m<2$, and $f\in H^2\cap L^1$. Then
  \begin{align}\label{Position-Lm}
    \sup_{t\geqslant0}\langle t\rangle^{1-\frac{1}{m}}\|K_0(t)f(\cdot)\|_{L^m}
    \leqslant C_m\big(\|f\|_{H^2}+\|f\|_{L^1}\big).
  \end{align}
  Moreover, $K_0(t)f\in\ml C\big([0,T],L^m\big)$ for every $T>0$, with $K_0(0)f=f$.
\end{lemma}

\begin{proof}
  We first consider the endpoint $m=\frac{10}{9}$. By the fractional Gagliardo--Nirenberg inequality in \cite[Theorem~1]{Brezis-Mironescu=2018} and the standard embedding $W^{\frac{2}{5},\frac{10}{9}}\hookrightarrow H_{10/9}^{\frac{2}{5}}$, we obtain
  \begin{align*}
    \|f\|_{H_{10/9}^{\frac{2}{5}}}
    &\lesssim\|f\|_{W^{\frac{2}{5},\frac{10}{9}}}
    \lesssim\|f\|_{L^1}^{\frac{4}{5}}\|f\|_{H^2}^{\frac{1}{5}}
    \lesssim\|f\|_{H^2}+\|f\|_{L^1}.
  \end{align*}
  Here, $H_p^s$ and $W^{s,p}$ denote the Bessel potential and Sobolev--Slobodeckij spaces, respectively.
  Applying \cite[Theorem~1.1, (1.8)--(1.9)]{Ikeda-Inui-Okamoto-Wakasugi=2019} to $K_0(t)=\partial_tK_1(t)+K_1(t)$, with input exponent $1$ and output exponent $\frac{10}{9}$, yields
  \begin{align*}
    \|K_0(t)f(\cdot)\|_{L^{\frac{10}{9}}}
    &\lesssim\langle t\rangle^{-\frac{1}{10}}\|f\|_{L^1}+\mathrm e^{-\frac{t}{2}}\langle t\rangle^d\|f\|_{H_{10/9}^{\frac{2}{5}}}\\
    &\lesssim\langle t\rangle^{-\frac{1}{10}}\big(\|f\|_{H^2}+\|f\|_{L^1}\big)
  \end{align*}
  for $t>0$, where $d>0$ is independent of $t$ and $f$.

  For $\frac{10}{9}<m<2$, one may choose $\theta\in(0,1)$ such that $\frac{1}{m}=\frac{9}{10}(1-\theta)+\frac{\theta}{2}$.
  Interpolating the preceding estimate with the $(L^2\cap L^1)$--$L^2$ decay estimate in \cite[Theorem~1.1]{Ikeda-Inui-Okamoto-Wakasugi=2019}, we obtain
  \begin{align*}
    \|K_0(t)f(\cdot)\|_{L^m}
    &\leqslant\|K_0(t)f(\cdot)\|_{L^{\frac{10}{9}}}^{1-\theta}
    \|K_0(t)f(\cdot)\|_{L^2}^{\theta}\\
    &\lesssim_m \langle t\rangle^{-1+\frac{1}{m}}\big(\|f\|_{H^2}+\|f\|_{L^1}\big).
  \end{align*}
  This proves \eqref{Position-Lm} for $t>0$, while its value at $t=0$ follows from $K_0(0)f=f$ and $L^1$--$L^2$ interpolation.

  For smooth compactly supported data, finite propagation and the $L^2$ continuity of the linear solution imply continuity in $L^m$ on $[0,T]$.
  The density of $\ml C_c^\infty$ in $H^2\cap L^1$, together with the uniform estimate \eqref{Position-Lm}, extends this continuity to every $f\in H^2\cap L^1$, including at $t=0$.
\end{proof}

\begin{prop}[$L^m$ continuity of the local solution]\label{Prop-Lm-Continuity}
  Let $n=2$, let $F$ satisfy Assumption~\ref{Assum-Chen}, and let $(u_0,u_1)\in\ml A_2$.
  Then the maximal mild Sobolev solution from Lemma~\ref{Lemma-Local} satisfies
  \begin{align*}
    u\in\ml C\big([0,T],L^m\big)
    \ \ \mbox{for every}\ \ 0<T<T_{\max},
  \end{align*}
  with $\frac{10}{9}\leqslant m<2$ carrying $u(0,\cdot)=u_0$ in $L^m$.
\end{prop}

\begin{proof}
  Fix $0<T<T_{\max}$.
  Lemma~\ref{Lemma-Position-Lm} gives the asserted continuity of $K_0(t)u_0$.
  To treat $K_1(t)u_1$ and the Duhamel term, let $1<m<2$.
  By the kernel formula in Lemma~\ref{lem:kernel}, for $0<t\leqslant1$, polar coordinates and the change of variables $q=\sqrt{t^2-r^2}$ yield
  \begin{align*}
    \|k_2(t,\cdot)\|_{L^m}^m
    &=(2\pi)^{1-m}\mathrm e^{-\frac{mt}{2}}
    \int_0^t\left[\cosh\left(\frac{q}{2}\right)\right]^m
    q^{1-m}\dd q\\
    &\lesssim_m\int_0^t q^{1-m}\dd q
    =\frac{t^{2-m}}{2-m}.
  \end{align*}
  After scaling to the unit ball, dominated convergence and the continuity of dilations in $L^m$ give
  $k_2\in\ml C\big((0,\infty),L^m\big)$, since $m<2$ makes the boundary singularity integrable.
  The preceding estimate extends this continuity to $t=0$ by setting $k_2(0,\cdot)=0$.
  Since $u_1\in L^1$, Young's inequality yields
  \begin{align*}
    K_1(t)u_1\in\ml C\big([0,T],L^m\big)
    \ \ \mbox{with}\ \ K_1(0)u_1=0.
  \end{align*}

  As in the proof of Lemma~\ref{Lemma-Local}, the mean value theorem, together with $\ml F_e'(0)=0$ and the local Sobolev regularity, gives
  \begin{align*}
    \big\|\ml F_e\big(u(t,\cdot)\big)-\ml F_e\big(u(s,\cdot)\big)\big\|_{L^1}
    &\leqslant C_T\big(\|u(t,\cdot)\|_{L^2}+\|u(s,\cdot)\|_{L^2}\big)
    \|u(t,\cdot)-u(s,\cdot)\|_{L^2}.
  \end{align*}
  Since $\ml F_e(0)=\ml F_e'(0)=0$ and $\ml F_e''$ is bounded on the attained amplitude range, we have
  \begin{align*}
    \big\|\ml F_e\big(u(t,\cdot)\big)\big\|_{L^1}
    \leqslant C_T\|u(t,\cdot)\|_{L^2}^2<\infty.
  \end{align*}
  The preceding difference estimate and the $L^2$ continuity of $u$ then yield
  \begin{align*}
    F(|u|)\in\ml C\big([0,T],L^1\big).
  \end{align*} 
  By the preceding continuity of $k_2$, Young's inequality and dominated convergence, the Duhamel term belongs to $\ml C\big([0,T],L^m\big)$ for every $1<m<2$, with value zero at $t=0$.
  Combining the three terms proves the asserted continuity and initial trace of $u$.
\end{proof}

\subsection{Amplitude-layer and Hardy--Stieltjes estimates}

Throughout this subsection, let $q>1$ and let $F:[0,\delta_0]\to[0,\infty)$ be continuous and nondecreasing with $F(0)=0$. To express the nonlinear source bounds through amplitude levels, we define, for $\ell\geqslant1$ and $0<a\leqslant\delta_0$,
\begin{align*}
  \ml T_\ell[F](a):=\int_{(0,a]}s^{-\ell}\dd F(s) \ \ \mbox{with} \ \ \ml T_\ell[F](0):=0,
\end{align*}
and introduce the $q$-critical Stieltjes functional
\begin{align}\label{Lambda-Definition}
  \Lambda_F(a):=\ml T_q[F](a)=\int_{(0,a]}s^{-q}\dd F(s) \ \ \mbox{with}\ \ \Lambda_F(0):=0.
\end{align}

\subsubsection{Critical Stieltjes functionals and Hardy estimates}

The functional in \eqref{Lambda-Definition} can be written in terms of the critical Dini-type integral.

\begin{lemma}[Stieltjes integration by parts]
  Assume that
  \begin{align}\label{Generic-Dini}
    \int_0^{\delta_0}\frac{F(s)}{s^{q+1}}\dd s<\infty.
  \end{align}
  Then 
  \begin{align}\label{F-Little-o}
    \lim_{a\to0^+}\frac{F(a)}{a^q}=0.
  \end{align}
  Moreover, for every $0<a\leqslant\delta_0$, one has
  \begin{align}\label{Lambda-Identity}
    \Lambda_F(a)=\frac{F(a)}{a^q} +q\int_0^a\frac{F(s)}{s^{q+1}}\dd s.
  \end{align}
  In particular,
  \begin{align}\label{Lambda-Vanish}
    \lim_{a\to0^+}\Lambda_F(a)=0.
  \end{align}
\end{lemma}

\begin{proof}
  Let $0<a<\frac{1}{2}\delta_0$. Since $F$ is nondecreasing, we estimate
  \begin{align*}
    \int_a^{2a}\frac{F(s)}{s^{q+1}}\dd s  \geqslant F(a)\int_a^{2a}s^{-q-1}\dd s  =\frac{1-2^{-q}}{q}\frac{F(a)}{a^q}.
  \end{align*}
  The left-hand side tends to zero as $a\to0^+$ by the absolute continuity of the integral in \eqref{Generic-Dini}. This proves \eqref{F-Little-o}.

  For $0<\varepsilon<a$, Stieltjes integration by parts gives
  \begin{align*}
    \int_{(\varepsilon,a]}s^{-q}\dd F(s)  =\frac{F(a)}{a^q}-\frac{F(\varepsilon)}{\varepsilon^q}  +q\int_\varepsilon^a\frac{F(s)}{s^{q+1}}\dd s.
  \end{align*}
  Passing to the limit $\varepsilon\to0^+$ and using \eqref{F-Little-o}, we obtain \eqref{Lambda-Identity}. The limit in \eqref{Lambda-Vanish} follows.
\end{proof}

\begin{lemma}[Hardy--Stieltjes estimate]\label{Lemma-Hardy}
  Let \eqref{Generic-Dini} hold and let $1\leqslant\ell<q$. Then
  \begin{align}\label{Hardy-Stieltjes}
    \int_0^a\rho^{\ell-q-1}\ml T_\ell[F](\rho)\dd\rho\leqslant\frac{1}{q-\ell}\Lambda_F(a)
  \end{align}
  for every $0<a\leqslant\delta_0$.
\end{lemma}

\begin{proof}
  Since all integrands are nonnegative, Tonelli's theorem gives
  \begin{align*}
    \int_0^a\rho^{\ell-q-1}\ml T_\ell[F](\rho)\dd\rho
    &=\int_{(0,a]}s^{-\ell}\left(\int_s^a\rho^{\ell-q-1}\dd\rho\right)\mathrm{d} F(s)\\
    &=\frac{1}{q-\ell}\int_{(0,a]}\left(s^{-q}-a^{\ell-q}s^{-\ell}\right)\mathrm{d} F(s)\leqslant\frac{1}{q-\ell}\Lambda_F(a),
  \end{align*}
  which proves the lemma.
\end{proof}

\subsubsection{Amplitude-layer inequalities}

We estimate the whole nonlinear source by integrating its amplitude levels.
This yields the source bounds used in place of \eqref{Freezing-Step}.

\begin{lemma}[Amplitude-layer estimates]\label{Lemma-Layer}
  Let $1\leqslant m<\infty$, $1\leqslant r\leqslant\infty$, and let $v\in L^m\cap L^r\cap L^\infty$ satisfy $\|v\|_{L^\infty}\leqslant\delta_0$. Then
  \begin{align}
    \|F(|v|)\|_{L^1}&\leqslant\|v\|_{L^m}^m\ml T_m[F](\|v\|_{L^\infty}),\label{Layer-L1}\\
    \|F(|v|)\|_{L^r}&\leqslant\|v\|_{L^r}\ml T_1[F](\|v\|_{L^\infty}).\label{Layer-Lr}
  \end{align}
\end{lemma}

\begin{proof}
  The estimates are immediate if $v=0$ almost everywhere, so we assume $\|v\|_{L^\infty}>0$. By the layer-cake representation and Tonelli's theorem, we arrive at
  \begin{align*}
    \|F(|v|)\|_{L^1}=\int_{\mb{R}^n}\int_{(0,|v(x)|]}\dd F(s)\dd x=\int_{(0,\|v\|_{L^\infty}]}\bigl|\{x:|v(x)|\geqslant s\}\bigr|\dd F(s).
  \end{align*}
  Chebyshev's inequality yields
  \begin{align*}
    \bigl|\{x:|v(x)|\geqslant s\}\bigr|  \leqslant s^{-m}\|v\|_{L^m}^m.
  \end{align*}
  Substitution into the preceding identity proves \eqref{Layer-L1}.

  If $\ml T_1[F](\|v\|_{L^\infty})=\infty$, then \eqref{Layer-Lr} is immediate.
  Otherwise, for $0\leqslant y\leqslant\|v\|_{L^\infty}$, we estimate
  \begin{align*}
    F(y)=\int_{(0,y]}\dd F(s)  \leqslant y\int_{(0,y]}s^{-1}\dd F(s)  \leqslant y\ml T_1[F](\|v\|_{L^\infty}).
  \end{align*}
  Applying this pointwise with $y=|v(x)|$ proves \eqref{Layer-Lr}.
\end{proof}

\subsection{Nonlinear source estimates}

For the small data argument, we set $q:=p_{\mathrm F}(n)=1+\frac{2}{n}$.
To close the nonlinear estimates, we fix the auxiliary exponent $m_*:=1+\frac{1}{n}$, so that $1<m_*<q$ for $n\in\{1,2\}$.

Under the assumptions of Theorem~\ref{Thm-Main}, let $u$ be the maximal mild Sobolev solution given by Lemma~\ref{Lemma-Local}. For $T\in(0,T_{\max})$, we define the following norm with the linear decay weights:
\begin{align}\label{X-Norm}
  \|u\|_{X(T)} :=\sup_{0\leqslant t\leqslant T}\bigg( \langle t\rangle^{\frac{n}{2(n+1)}}\|u(t,\cdot)\|_{L^{m_*}} +\sum_{j=0}^1\langle t\rangle^{\frac{n}{4}+\frac{j}{2}} \|\nabla^j u(t,\cdot)\|_{L^2}+\langle t\rangle^{\frac{n}{2}}\|u(t,\cdot)\|_{L^\infty} \bigg).
\end{align}
The local Sobolev regularity and, for $n=2$, Proposition~\ref{Prop-Lm-Continuity} ensure that $\|u\|_{X(T)}<\infty$ for every $T<T_{\max}$.
We denote the linear part by
\begin{align*}
  u^{\mathrm{lin}}(t,x):=\big(K_0(t)u_0\big)(x)+\big(K_1(t)u_1\big)(x).
\end{align*}
By \cite[Lemma 1]{Matsumura=1976}, Lemma~\ref{Lemma-Position-Lm}, and \eqref{Linear-L1-Lm}, we obtain
\begin{align*}
  \|u^{\mathrm{lin}}\|_{X(\infty)}
  \lesssim\|(u_0,u_1)\|_{\ml A_n}.
\end{align*}
We write
\begin{align}\label{M-Definition}
  A(t):=\|u\|_{X(T)}\langle t\rangle^{-\frac{n}{2}}.
\end{align}
In the bootstrap argument below, $\|u\|_{X(T)}$ is sufficiently small so that $A(t)\leqslant\delta_0$.

To control the time integral of the $L^1$ source norm, we define the nonincreasing envelope
\begin{align}\label{g-Definition}
  g(t):=\|u\|_{X(T)}^{m_*}\langle t\rangle^{-\frac{1}{2}} \ml T_{m_*}[F]\big(A(t)\big).
\end{align}
By \eqref{X-Norm}, \eqref{M-Definition}, and Lemma~\ref{Lemma-Layer} with $m=m_*$, we obtain
\begin{align}\label{Source-L1-Pointwise}
  \big\|F\big(|u(t,\cdot)|\big)\big\|_{L^1}\leqslant g(t)
  \ \ \mbox{for}\ \ 0\leqslant t\leqslant T.
\end{align}

The amplitude change of variables gives
\begin{align}\label{Amplitude-Change}
  a:=A(t)=\|u\|_{X(T)}\langle t\rangle^{-\frac{n}{2}}\ \ \Rightarrow \ \ \dd t =-\frac{2}{n}\|u\|_{X(T)}^{q-1}a^{-q}\dd a.
\end{align}
Therefore, by $1<m_*<q$, $q-m_*=\frac{1}{n}$, and \eqref{Hardy-Stieltjes} with $\ell=m_*$, we obtain
\begin{align}\label{Source-L1-Integral}
  \int_0^\infty g(t)\dd t
  &=\frac{2\|u\|_{X(T)}^q}{n}\int_0^{\|u\|_{X(T)}} a^{m_*-q-1}\ml T_{m_*}[F](a)\dd a\notag\\
  &\leqslant2\|u\|_{X(T)}^q\Lambda_F\big(\|u\|_{X(T)}\big).
\end{align}

To estimate the source in $L^2$ and $L^{2n}$, let us define
\begin{align*}
  h(t):=\ml T_1[F]\big(A(t)\big),
\end{align*}
which is nonincreasing. From Lemma~\ref{Lemma-Layer}, we obtain
\begin{align*}
  \big\|F\big(|u(t,\cdot)|\big)\big\|_{L^\ell}\leqslant h(t)\|u(t,\cdot)\|_{L^\ell}
\end{align*}
for every $m_*\leqslant\ell\leqslant\infty$ and $0\leqslant t\leqslant T$. In particular,
\begin{align}\label{Source-L2}
  \big\|F\big(|u(t,\cdot)|\big)\big\|_{L^2}\lesssim \|u\|_{X(T)}\langle t\rangle^{-\frac{n}{4}}h(t).
\end{align}
Moreover, \eqref{Amplitude-Change} and \eqref{Hardy-Stieltjes} with $\ell=1$ imply
\begin{align}\label{h-Integral}
  \int_0^\infty h(t)\dd t&=\frac{2\|u\|_{X(T)}^{q-1}}{n}\int_0^{\|u\|_{X(T)}} a^{-q}\ml T_1[F](a)\dd a\notag\\
  &\lesssim \|u\|_{X(T)}^{q-1}\Lambda_F\big(\|u\|_{X(T)}\big).
\end{align}
Since $h$ is nonincreasing, for every $t\geqslant2$, we have
\begin{align}\label{h-half}
  h\left(\frac{t}{2}\right)
  \leqslant\frac{2}{t}\int_0^{\frac{t}{2}}h(s)\dd s
  \lesssim t^{-1}\|u\|_{X(T)}^{q-1}\Lambda_F\big(\|u\|_{X(T)}\big).
\end{align}
Interpolation between $L^{m_*}$ and $L^\infty$ yields
\begin{align*}
  \|u(t,\cdot)\|_{L^{2n}}\leqslant\|u(t,\cdot)\|_{L^{m_*}}^{\frac{m_*}{2n}} \|u(t,\cdot)\|_{L^\infty}^{1-\frac{m_*}{2n}}\lesssim \|u\|_{X(T)}\langle t\rangle^{-\frac{n}{2}+\frac{1}{4}}.
\end{align*}
Consequently,
\begin{align}\label{Source-Lr}
  \big\|F\big(|u(t,\cdot)|\big)\big\|_{L^{2n}}\lesssim \|u\|_{X(T)}\langle t\rangle^{-\frac{n}{2}+\frac{1}{4}}h(t).
\end{align}
Integrating \eqref{Source-L2} and \eqref{Source-Lr} over $[0,T]$ and applying \eqref{h-Integral}, we obtain
\begin{align}\label{Source-High-Integral}
  \int_0^T\Big(\big\|F\big(|u(t,\cdot)|\big)\big\|_{L^2}+\big\|F\big(|u(t,\cdot)|\big)\big\|_{L^{2n}}\Big)\dd t\lesssim \|u\|_{X(T)}^q\Lambda_F\big(\|u\|_{X(T)}\big).
\end{align}

\subsection{Duhamel estimates: Proof of Theorem~\ref{Thm-Main}}
The following useful estimate follows by splitting the time integral at $\frac{t}{2}$ and using the monotonicity of $g$.
\begin{lemma}[Convolution with a nonincreasing integrable envelope]\label{Lemma-Convolution}
  Let $0\leqslant\alpha<1$, and let $g:[0,\infty)\to[0,\infty)$ be nonincreasing and integrable. Then
  \begin{align}\label{Convolution-Envelope}
    \int_0^t\langle t-s\rangle^{-\alpha}g(s)\dd s \lesssim\langle t\rangle^{-\alpha}\|g\|_{L^1(0,\infty)}
  \end{align}
  for every $t\geqslant0$.
\end{lemma}

We define the nonlinear Duhamel term by
\begin{align*}
  G[u](t,x):=\int_0^t\big(K_1(t-s)F\big(|u(s,\cdot)|\big)\big)(x)\dd s.
\end{align*}
The estimates in \eqref{Linear-L1-Lm}--\eqref{Linear-Lr-Linfty}, together with \eqref{Source-L1-Integral} and \eqref{Source-High-Integral}, give
\begin{align}\label{NEW}
  \|G[u]\|_{X(\min\{2,T\})}\lesssim \|u\|_{X(T)}^q\Lambda_F\big(\|u\|_{X(T)}\big).
\end{align}
It remains to derive the weighted estimates for $2\leqslant t\leqslant T$.

By applying \eqref{Linear-L1-Lm}, \eqref{Source-L1-Pointwise}, Lemma~\ref{Lemma-Convolution} and \eqref{Source-L1-Integral}, we obtain
\begin{align}\label{Duhamel-Lm}
  \|G[u](t,\cdot)\|_{L^{m_*}}
  &\lesssim\int_0^t\langle t-s\rangle^{-\frac{n}{2(n+1)}}g(s)\dd s\notag\\
  &\lesssim\langle t\rangle^{-\frac{n}{2(n+1)}}
  \|u\|_{X(T)}^q\Lambda_F\big(\|u\|_{X(T)}\big).
\end{align}
We split $G[u](t,x)$ into the following two parts:
\begin{align*}
  G^{(0)}[u](t,x)&:=\int_0^{\frac{t}{2}}\big(K_1(t-s)F\big(|u(s,\cdot)|\big)\big)(x)\dd s,\\
  G^{(1)}[u](t,x)&:=\int_{\frac{t}{2}}^t\big(K_1(t-s)F\big(|u(s,\cdot)|\big)\big)(x)\dd s.
\end{align*}
For $j\in\{0,1\}$, applying \eqref{Linear-Energy-Mixed} on the early interval gives
\begin{align}\label{Duhamel-Energy-Early}
  \left\|\nabla^jG^{(0)}[u](t,\cdot)\right\|_{L^2}
  &\lesssim t^{-\frac{n}{4}-\frac{j}{2}}\int_0^{\frac{t}{2}}\big\|F\big(|u(s,\cdot)|\big)\big\|_{L^1}\dd s+\mathrm e^{-\frac{ct}{2}}\int_0^{\frac{t}{2}}\big\|F\big(|u(s,\cdot)|\big)\big\|_{L^2}\dd s\notag\\
  &\lesssim t^{-\frac{n}{4}-\frac{j}{2}}\|u\|_{X(T)}^q\Lambda_F\big(\|u\|_{X(T)}\big).
\end{align}
For $s\in[\frac{t}{2},t]$, \eqref{Linear-L2-L2}, \eqref{Source-L2}, and the monotonicity of $h$ give
\begin{align}
  \left\|\nabla^jG^{(1)}[u](t,\cdot)\right\|_{L^2}&\lesssim \|u\|_{X(T)}\langle t\rangle^{-\frac{n}{4}}h\left(\frac{t}{2}\right) \int_0^{\frac{t}{2}}\langle\tau\rangle^{-\frac{j}{2}}\dd\tau\notag\\
  &\lesssim\langle t\rangle^{-\frac{n}{4}-\frac{j}{2}}\|u\|_{X(T)}^q\Lambda_F\big(\|u\|_{X(T)}\big),\label{Duhamel-Energy-Late}
\end{align}
where we used \eqref{h-half}. Combining \eqref{Duhamel-Energy-Early} and \eqref{Duhamel-Energy-Late}, we obtain
\begin{align}\label{Duhamel-Energy}
  \|\nabla^jG[u](t,\cdot)\|_{L^2}\lesssim\langle t\rangle^{-\frac{n}{4}-\frac{j}{2}}\|u\|_{X(T)}^q\Lambda_F\big(\|u\|_{X(T)}\big) \ \ \mbox{for}\ \ j\in\{0,1\}.
\end{align}
Applying \eqref{Linear-Linfty-Mixed} with exponent $2n$ yields
\begin{align*}
  \left\|G^{(0)}[u](t,\cdot)\right\|_{L^\infty}&\lesssim t^{-\frac{n}{2}}\int_0^{\frac{t}{2}}\big\|F\big(|u(s,\cdot)|\big)\big\|_{L^1}\dd s+\mathrm e^{-c_1t}\int_0^{\frac{t}{2}}\big\|F\big(|u(s,\cdot)|\big)\big\|_{L^{2n}}\dd s\\
  &\lesssim t^{-\frac{n}{2}}\|u\|_{X(T)}^q\Lambda_F\big(\|u\|_{X(T)}\big).
\end{align*}
For the late-time interval, we apply \eqref{Linear-Lr-Linfty} with exponent $2n$ and \eqref{Source-Lr}. Since $h$ is nonincreasing, one has
\begin{align*}
  \left\|G^{(1)}[u](t,\cdot)\right\|_{L^\infty}&\lesssim \|u\|_{X(T)}\langle t\rangle^{-\frac{n}{2}+\frac{1}{4}} h\left(\frac{t}{2}\right) \int_0^{\frac{t}{2}}\langle\tau\rangle^{-\frac{1}{4}}\dd\tau\\
  &\lesssim\langle t\rangle^{-\frac{n}{2}}\|u\|_{X(T)}^q\Lambda_F\big(\|u\|_{X(T)}\big).
\end{align*}
Therefore,
\begin{align}\label{Duhamel-Linfty}
  \|G[u](t,\cdot)\|_{L^\infty} \lesssim\langle t\rangle^{-\frac{n}{2}}\|u\|_{X(T)}^q\Lambda_F\big(\|u\|_{X(T)}\big).
\end{align}

Combining the estimates for the linear problem, \eqref{NEW}, \eqref{Duhamel-Lm}, \eqref{Duhamel-Energy}, and \eqref{Duhamel-Linfty}, we obtain
\begin{align}\label{Bootstrap-Inequality}
  \|u\|_{X(T)} \leqslant C_0\|(u_0,u_1)\|_{\ml A_n} +C\|u\|_{X(T)}^q\Lambda_F\bigl(\|u\|_{X(T)}\bigr),
\end{align}
provided that $\|u\|_{X(T)}\leqslant\delta_0$.
Enlarging $C_0$ once, if necessary, we may also bound the initial value of the $X(T)$ norm by $C_0\|(u_0,u_1)\|_{\ml A_n}$.

Let $\|(u_0,u_1)\|_{\ml A_n}\leqslant\varepsilon$ and assume the bootstrap bound
\begin{align}\label{Bootstrap-Assumption}
  \|u\|_{X(T)}\leqslant2C_0\varepsilon.
\end{align}
By \eqref{Lambda-Vanish}, we may choose $\varepsilon_0>0$ sufficiently small so that
\begin{align*}
  2C_0\varepsilon_0\leqslant\delta_0\ \ \mbox{as well as}\ \  C(2C_0\varepsilon_0)^{q-1}\Lambda_F(2C_0\varepsilon_0)\leqslant\frac{1}{4}.
\end{align*}
Since $\Lambda_F$ is nondecreasing, the same bounds hold for every $0<\varepsilon\leqslant\varepsilon_0$. Then \eqref{Bootstrap-Inequality} improves \eqref{Bootstrap-Assumption} to
\begin{align*}
  \|u\|_{X(T)}\leqslant\frac{3}{2}C_0\varepsilon.
\end{align*}
The local solution is continuous in the spaces defining $X(T)$, including $L^{m_*}$ in dimension two by Proposition~\ref{Prop-Lm-Continuity}.
Thus $T\mapsto\|u\|_{X(T)}$ is continuous, with initial value at most $C_0\varepsilon$.
If it first reached $2C_0\varepsilon$ at some $T<T_{\max}$, the improved bound would give a contradiction.
Consequently, the bootstrap bound holds throughout the maximal interval, and
\begin{align*}
  \sup_{0\leqslant t<T_{\max}}\|u(t,\cdot)\|_{L^\infty} \leqslant2C_0\varepsilon<\infty.
\end{align*}
The continuation criterion \eqref{Continuation-Criterion} implies $T_{\max}=\infty$. 

The $L^2$-based and $L^\infty$ estimates in Theorem~\ref{Thm-Main} follow from the definition of $X(T)$, which also proves the $L^m$ estimate for $n=1$.
For $n=2$, replacing $\|u\|_{X(T)}$ by $\|u\|_{X(\infty)}$ in the source envelope, we obtain from \eqref{Linear-L1-Lm}, Lemma~\ref{Lemma-Convolution}, and \eqref{Source-L1-Integral} that
\begin{align*}
  \|G[u](t,\cdot)\|_{L^m}
  &\lesssim_m\langle t\rangle^{-1+\frac{1}{m}}
  \|u\|_{X(\infty)}^2\Lambda_F\big(\|u\|_{X(\infty)}\big)
  \lesssim_m\varepsilon\langle t\rangle^{-1+\frac{1}{m}}
  \ \ \mbox{for}\ \ 1<m<2.
\end{align*}
Together with \eqref{Position-Lm} and \eqref{Linear-L1-Lm}, this gives the asserted $L^m$ estimate for $\frac{10}{9}\leqslant m<2$ and completes the proof.

\section{Growth of global solutions under Dini-type divergence}\label{Sec-Growth}

\subsection{The positive velocity kernel and signed initial data}

Recall that $K_1(t)$ is the velocity propagator for the homogeneous damped wave equation and that $K_0(t)=\partial_tK_1(t)+K_1(t)$ is its position propagator.
Their Fourier multipliers are
\begin{align}\label{eq:multipliers}
  m_1(t,\xi)=\mathrm e^{-\frac{t}{2}}
  \frac{\sinh\left(t\sqrt{\frac{1}{4}-|\xi|^2}\,\right)}{\sqrt{\frac{1}{4}-|\xi|^2}}
  \ \ \text{and}\ \ 
  m_0(t,\xi)=\partial_tm_1(t,\xi)+m_1(t,\xi).
\end{align}
The quotient is interpreted by continuity when $|\xi|=\frac{1}{2}$ and by
the sine formula when $|\xi|>\frac{1}{2}$.

\begin{lemma}[Positive velocity kernels]\label{lem:kernel}
  For $n\in\{1,2\}$, the operator $K_1(t)$ is convolution with the nonnegative
  integrable function $k_n(t,\cdot)$ given, on $|x|<t$, by
  \begin{align}\label{eq:kernels}
    k_1(t,x)=\frac{\mathrm e^{-\frac{t}{2}}}{2}
    I_0\left(\frac{\sqrt{t^2-x^2}}{2}\right)
    \ \ \text{and}\ \ 
    k_2(t,x)=\frac{\mathrm e^{-\frac{t}{2}}}{2\pi}
    \frac{\cosh\left(\frac{\sqrt{t^2-|x|^2}}{2}\right)}{\sqrt{t^2-|x|^2}}.
  \end{align}
  Here, $I_0$ is the modified Bessel function of the first kind of order zero.
  Both kernels vanish outside $|x|<t$, and
  \begin{align}\label{eq:kernel-mass}
    \|k_n(t,\cdot)\|_{L^1}=1-\mathrm e^{-t}.
  \end{align}
  For each $A>0$, there exists $c_{n,A}>0$ such that
  \begin{align}\label{eq:core-kernel}
    k_n(\tau,z)\geqslant c_{n,A}\tau^{-\frac{n}{2}}
    \ \ \text{for}\ \ \tau\geqslant \tau_A:=\max\{2A^2,4\},
    \ \ |z|\leqslant A\sqrt{\tau}.
  \end{align}
\end{lemma}

\begin{proof}
  The kernel formulae follow from \cite[Proposition 2.1 and Example 2.2]{Sakata-Wakasugi=2017} with $n\in\{1,2\}$.
  They define nonnegative integrable kernels supported in $\{|x|\leqslant t\}$.
  The Fourier multiplier formula gives the total mass of the kernel:
  \begin{align*}
    \int_{\mb R^n}k_n(t,x)\dd x
    =m_1(t,0)
    =2\mathrm e^{-\frac{t}{2}}\sinh\left(\frac{t}{2}\right)
    =1-\mathrm e^{-t}.
  \end{align*}

  To prove \eqref{eq:core-kernel}, let $\tau\geqslant \tau_A$ and $|z|\leqslant A\sqrt{\tau}$, and write $q:=\sqrt{\tau^2-|z|^2}$.
  Since $\tau_A=\max\{2A^2,4\}$, we have
  \begin{align*}
    \frac{\tau}{\sqrt2}\leqslant q\leqslant\tau
    \ \ \text{and}\ \ 
    \frac{\tau-q}{2}
    =\frac{|z|^2}{2(\tau+q)}\leqslant\frac{A^2}{2}.
  \end{align*}
  \begin{itemize}
    \item For $n=1$, we use the integral representation $I_0(\rho)=\frac{1}{\pi}\int_0^\pi\mathrm e^{\rho\cos\theta}\dd\theta$.
    Shrinking the domain of integration to $[0,\rho^{-\frac{1}{2}}]$ and using $\cos\theta\geqslant1-\frac{\theta^2}{2}$, we obtain
    \begin{align*}
      I_0(\rho)\geqslant\frac{1}{\pi}\int_0^{\rho^{-\frac{1}{2}}}
      \mathrm e^{\rho\big(1-\frac{\theta^2}{2}\big)}\dd\theta\geqslant\frac{\mathrm e^{\rho-\frac{1}{2}}}{\pi\sqrt{\rho}}
      \ \ \text{for}\ \ \rho\geqslant1.
    \end{align*}
    Recalling that $\tau\geqslant \tau_A=\max\{2A^2,4\}$, the preceding bound $q\geqslant\frac{\tau}{\sqrt2}$ gives $q\geqslant2$.
    Taking $\rho=\frac{q}{2}\geqslant1$ in the preceding estimate yields
    \begin{align*}
      k_1(\tau,z)\geqslant\frac{\mathrm e^{-\frac{\tau-q+1}{2}}}{\sqrt2\pi\sqrt q}\geqslant\frac{\mathrm e^{-\frac{A^2+1}{2}}}{\sqrt2\pi}\tau^{-\frac{1}{2}}.
    \end{align*}
    \item For $n=2$, the fact $\cosh\big(\frac{q}{2}\big)\geqslant\frac{1}{2}\mathrm e^{\frac{q}{2}}$ yields
    \begin{align*}
      k_2(\tau,z)\geqslant\frac{\mathrm e^{-\frac{\tau-q}{2}}}{4\pi q}\geqslant\frac{\mathrm e^{-\frac{A^2}{2}}}{4\pi}\tau^{-1}.
    \end{align*}
  \end{itemize}
  These estimates prove \eqref{eq:core-kernel} with $c_{n,A}>0$ depending only on $n$ and $A$.
\end{proof}

The diffusion asymptotic of the full linear solution gives a positive lower bound on expanding balls for data with positive effective total mass.

\begin{lemma}[Positive linear lower bound]\label{lem:linear}
  Let $(u_0,u_1)\in H^{N_n}\times H^{N_n-1}$, and let $u^{\mathrm{lin}}$ denote the linear solution introduced in Section~\ref{Sec-Small}.
  There is a finite constant $L\geqslant0$, depending only on the data and $n$,
  such that
  \begin{align}\label{eq:linear-bound}
    |u^{\mathrm{lin}}(t,x)|\leqslant L\ \ \text{for}\ \ t\geqslant0,\ x\in\mb R^n.
  \end{align}
  If, in addition, $(u_0,u_1)\in\ml A_n$ and \eqref{Effective-Mass} holds, there are $a>0$ and $t_0\geqslant18$ such that
  \begin{align}\label{eq:linear-seed}
    u^{\mathrm{lin}}(t,x)\geqslant a t^{-\frac{n}{2}}
    \ \ \text{for}\ \ t\geqslant t_0,\ |x|\leqslant\sqrt t.
  \end{align}
\end{lemma}

\begin{proof}
  For Sobolev data, the Fourier multipliers satisfy, uniformly for $t\geqslant0$,
  \begin{align}\label{eq:uniform-multipliers}
    |m_0(t,\xi)|\leqslant C
    \ \ \text{and}\ \ 
    |m_1(t,\xi)|\leqslant C\langle\xi\rangle^{-1}.
  \end{align}
  Here, $\langle\xi\rangle:=(1+|\xi|^2)^{\frac{1}{2}}$.
  On $|\xi|\leqslant\frac{1}{4}$, these bounds follow from the two real roots.
  On $\frac{1}{4}\leqslant|\xi|\leqslant1$, a bound
  $C(1+t)\mathrm e^{-ct}$ removes the apparent singularity at
  $|\xi|=\frac{1}{2}$.
  On $|\xi|\geqslant1$, the sine formula gives the additional factor
  $|\xi|^{-1}$ in $m_1$.
  Therefore, we obtain
  \begin{align*}
    \sup_{t\geqslant0}\|u^{\mathrm{lin}}(t,\cdot)\|_{H^{N_n}}
    \leqslant C\big(\|u_0\|_{H^{N_n}}+\|u_1\|_{H^{N_n-1}}\big).
  \end{align*}
  Since $N_n>\frac{n}{2}$, Sobolev embedding proves \eqref{eq:linear-bound}.

  We also give a quantitative error estimate for the diffusion comparison.
  For $r:=|\xi|\leqslant\frac{1}{4}$, we write $d:=\sqrt{1-4r^2}$ and $\lambda_\pm:=\frac{-1\pm d}{2}$.
  Then
  \begin{align*}
    m_1(t,\xi)&=\frac{\mathrm e^{\lambda_+t}
    -\mathrm e^{\lambda_-t}}{d}
    \ \ \text{and}\ \  m_0(t,\xi)=\frac{1+d}{2d}\mathrm e^{\lambda_+t}
    +\frac{d-1}{2d}\mathrm e^{\lambda_-t}.
  \end{align*}
  For $|\xi|\leqslant\frac{1}{4}$, we have $|\lambda_++r^2|\leqslant Cr^4$, $\lambda_+\leqslant-r^2$, and $\lambda_-\leqslant-c$. The coefficients of $\mathrm e^{\lambda_+t}$ differ from one by $O(r^2)$.
  Thus, for $j\in\{0,1\}$,
  \begin{align*}
    |m_j(t,\xi)-\mathrm e^{-tr^2}|
    \leqslant C(r^2+tr^4)\mathrm e^{-tr^2}+C\mathrm e^{-ct}.
  \end{align*}
  With $\eta=\sqrt t\,\xi$, the low-frequency integral satisfies
  \begin{align*}
    \int_{|\xi|\leqslant\frac{1}{4}}
    \big(|\xi|^2+t|\xi|^4\big)\mathrm e^{-t|\xi|^2}\dd\xi
    &\leqslant t^{-\frac{n}{2}-1}
    \int_{\mb R^n}\big(|\eta|^2+|\eta|^4\big)\mathrm e^{-|\eta|^2}\dd\eta
    \lesssim t^{-\frac{n}{2}-1}
  \end{align*}
  for $t\geqslant1$.
  On $|\xi|\geqslant\frac{1}{4}$, the preceding high-frequency estimates give
  \begin{align*}
    |m_0(t,\xi)|+\langle\xi\rangle|m_1(t,\xi)|
    \leqslant C\mathrm e^{-ct}
  \end{align*}
  after reducing $c$.
  Because of $N_n>\frac{n}{2}$, we can apply the Cauchy--Schwarz inequality with $\int_{\mb R^n}\langle\xi\rangle^{-2N_n}\dd\xi<\infty$.
  For the high-frequency heat term, we use $|\widehat{u_0+u_1}|\leqslant\|u_0+u_1\|_{L^1}$.
  Combining the low- and high-frequency estimates, we obtain
  \begin{align}\label{eq:diffusion-error}
    &\|u^{\mathrm{lin}}(t,\cdot)-\mathrm e^{t\Delta}(u_0+u_1)(\cdot)\|_{L^\infty}\notag\\
    &\leqslant C t^{-\frac{n}{2}-1}
    \big(\|u_0\|_{L^1}+\|u_1\|_{L^1}\big)+C\mathrm e^{-ct}
    \big(\|u_0\|_{H^{N_n}}+\|u_1\|_{H^{N_n-1}}\big)
  \end{align}
  for $t\geqslant1$.

  For $f=u_0+u_1\in L^1$, the explicit representation of $\mathrm e^{t\Delta}f$ gives
  \begin{align*}
    t^{\frac{n}{2}}(\mathrm e^{t\Delta}f)(\sqrt t\,z)
    =\frac{1}{(4\pi)^{\frac{n}{2}}}
    \int_{\mb R^n}\exp\left(-\frac{1}{4}\left|z-\frac{y}{\sqrt{t}}\right|^2\right)f(y)\dd y.
  \end{align*}
  At each fixed $y$, the integrand converges uniformly for $|z|\leqslant1$, and its difference from $\mathrm e^{-\frac{|z|^2}{4}}f(y)$ is bounded by $2|f(y)|$.
  Consequently, we obtain
  \begin{align*}
    &\sup_{|z|\leqslant1}
    \left|t^{\frac{n}{2}}(\mathrm e^{t\Delta}f)(\sqrt t\,z)
    -\frac{M_0}{(4\pi)^{\frac{n}{2}}}\mathrm e^{-\frac{|z|^2}{4}}\right|\leqslant\frac{1}{(4\pi)^{\frac{n}{2}}}
    \int_{\mb R^n}
    \sup_{|z|\leqslant1}
    \left|\mathrm e^{-\frac{1}{4}\big|z-\frac{y}{\sqrt t}\big|^2}
    -\mathrm e^{-\frac{|z|^2}{4}}\right|
    |f(y)|\dd y,
  \end{align*}
  which tends to $0$ as $t\to\infty$ by dominated convergence.
  For $|x|\leqslant\sqrt t$, the preceding convergence and \eqref{eq:diffusion-error} yield
  \begin{align*}
    t^{\frac{n}{2}}u^{\mathrm{lin}}(t,x)
    &\geqslant\frac{M_0}{(4\pi)^{\frac{n}{2}}}
    \mathrm e^{-\frac{|x|^2}{4t}}
    -t^{\frac{n}{2}}
    \|u^{\mathrm{lin}}(t,\cdot)-\mathrm e^{t\Delta}f\|_{L^\infty}-\sup_{|z|\leqslant1}
    \left|t^{\frac{n}{2}}(\mathrm e^{t\Delta}f)(\sqrt t\,z)
    -\frac{M_0}{(4\pi)^{\frac{n}{2}}}\mathrm e^{-\frac{|z|^2}{4}}\right|\\
    &\geqslant\frac{M_0}{(4\pi)^{\frac{n}{2}}}\mathrm e^{-\frac{1}{4}}
    -o(1)
    \geqslant\frac{M_0}{2(4\pi)^{\frac{n}{2}}}\mathrm e^{-\frac{1}{4}}
    =:a
  \end{align*}
  for all sufficiently large $t\geqslant t_0$, which proves \eqref{eq:linear-seed}.
\end{proof}

\begin{lemma}[Duhamel representation for weak solutions]\label{lem:weak-duhamel}
  Every weak solution in Definition~\ref{Defn-Weak} satisfies
  \begin{align}\label{eq:weak-duhamel}
    u(t,x)=u^{\mathrm{lin}}(t,x)+
    \int_0^t\int_{\mb R^n}
    k_n(t-s,x-y)F\big(|u(s,y)|\big)\dd y\dd s
  \end{align}
  for almost every $(t,x)$.
  In particular,
  \begin{align}\label{eq:negative-part}
    u(t,x)\geqslant u^{\mathrm{lin}}(t,x)\geqslant-L
  \end{align}
  almost everywhere, where $L$ is the constant in
  \eqref{eq:linear-bound}.
\end{lemma}

\begin{proof}
  \noindent\textbf{Step 1. Local integrability of the Duhamel term.}
  For $f\in L^1_{\mathrm{loc}}\big([0,T)\times\mb R^n\big)$, we consider
  \begin{align*}
    (Vf)(t,x):=\int_0^t\int_{\mb R^n}
    k_n(t-s,x-y)f(s,y)\dd y\dd s.
  \end{align*}
  To justify this integral, let us fix $0<T'<T$ and $R>0$.
  By finite propagation, $k_n(t-s,x-y)=0$ whenever $x\in B_R$, $y\notin B_{R+T'}$ and $0\leqslant s\leqslant t\leqslant T'$.
  By positivity, Tonelli's theorem and \eqref{eq:kernel-mass}, we obtain
  \begin{align}\label{eq:local-convolution}
    \int_0^{T'}\int_{B_R}(V|f|)(t,x)\dd x\dd t
    &=\int_0^{T'}\int_{B_{R+T'}}|f(s,y)|
    \left(\int_s^{T'}\int_{B_R}k_n(t-s,x-y)\dd x\dd t\right)\mathrm{d} y\dd s\notag\\
    &\leqslant\int_0^{T'}\int_{B_{R+T'}}|f(s,y)|
    \left(\int_s^{T'}\big(1-\mathrm e^{-(t-s)}\big)\dd t\right)\mathrm{d} y\dd s\notag\\
    &\leqslant T'\int_0^{T'}\int_{B_{R+T'}}|f(s,y)|\dd y\dd s.
  \end{align}
  Thus the integral defining $Vf$ is absolutely convergent almost everywhere, and $|Vf|\leqslant V|f|$ implies
  $Vf\in L^1_{\mathrm{loc}}\big([0,T)\times\mb R^n\big)$.

  \medskip
  \noindent\textbf{Step 2. The equation and initial data for the Duhamel term.}
  For smooth $f$, the multiplier formula \eqref{eq:multipliers} shows that $Vf$ solves the damped wave equation with source $f$ and zero initial data.
  Consequently, integration by parts yields
  \begin{align*}
    \int_0^T\int_{\mb R^n}(Vf)(t,x)
    \big(\varphi_{tt}-\Delta\varphi-\varphi_t\big)(t,x)\dd x\dd t
    =\int_0^T\int_{\mb R^n}f(t,x)\varphi(t,x)\dd x\dd t
  \end{align*}
  for every $\varphi\in\ml C_c^\infty\big([0,T)\times\mb R^n\big)$.
  By smooth approximation in $L^1_{\mathrm{loc}}$ and \eqref{eq:local-convolution}, this identity extends to every
  $f\in L^1_{\mathrm{loc}}\big([0,T)\times\mb R^n\big)$, including test functions which do not vanish at $t=0$.
  Taking $f=F(|u|)$, we conclude that $u^{\mathrm{lin}}+Vf$ has the same source and initial terms as $u$ in the weak formulation.

  \medskip
  \noindent\textbf{Step 3. Identification with the weak solution.}
  It remains to prove uniqueness for the linear equation in this locally integrable class.
  The difference $w:=u-u^{\mathrm{lin}}-Vf$ satisfies
  \begin{align*}
    \int_0^T\int_{\mb R^n}w(t,x)
    \big(\varphi_{tt}-\Delta\varphi-\varphi_t\big)(t,x)\dd x\dd t=0
  \end{align*}
  for every $\varphi\in\ml C_c^\infty\big([0,T)\times\mb R^n\big)$.
  To show that $w$ vanishes against an arbitrary
  $\chi\in\ml C_c^\infty\big((0,T')\times\mb R^n\big)$, we solve the backward adjoint equation by setting
  \begin{align*}
    \psi(t,x):=\int_t^{T'}\big(K_1(s-t)\chi(s,\cdot)\big)(x)\dd s.
  \end{align*}
  The propagator formula gives
  \begin{align*}
    \psi_{tt}-\Delta\psi-\psi_t=\chi
    \ \ \text{on}\ \ (0,T')\times\mb R^n.
  \end{align*}
  Smoothness of $\chi$ and finite propagation imply that $\psi$ is smooth up to $t=0$ and has compact spatial support.
  Moreover, $\psi$ vanishes near $T'$ because $\chi$ does.
  Extending $\psi$ by zero for $t\geqslant T'$, we may therefore use it as a test function in the preceding weak identity and obtain
  \begin{align*}
    0&=\int_0^{T'}\int_{\mb R^n}w(t,x)
    \big(\psi_{tt}-\Delta\psi-\psi_t\big)(t,x)\dd x\dd t\\
    &=\int_0^{T'}\int_{\mb R^n}w(t,x)\chi(t,x)\dd x\dd t.
  \end{align*}
  Since $\chi$ and $T'<T$ are arbitrary, we have $w=0$ as a distribution and almost everywhere.
  This proves the Duhamel representation, and the nonnegativity of $k_n$ and $F(|u|)$ yields \eqref{eq:negative-part}.
\end{proof}

By the local Sobolev regularity, the Duhamel formula and integration by parts in time, every mild Sobolev solution satisfies \eqref{Weak-Identity} with the prescribed initial data.
Thus we obtain the following lemma.

\begin{lemma}[Mild solutions are weak solutions]\label{lem:mild-weak}
  For data in $\ml A_n$, every mild Sobolev solution is a weak solution in the sense of Definition~\ref{Defn-Weak} on its existence interval.
\end{lemma}

\subsection{Persistent growth via delayed comparison}

By restricting the source times in the Duhamel representation of Lemma~\ref{lem:weak-duhamel} to $s\leqslant\frac{t}{2}$ and applying \eqref{eq:linear-seed} and \eqref{eq:core-kernel}, we obtain a delayed integral equation whose solution bounds $u$ from below on $\{|x|\leqslant\sqrt t\}$.
The Dini-type divergence condition \eqref{Intrinsic-NonDini} rules out bounded solutions of this auxiliary equation.
For each prescribed positive level, a further comparison argument shows that reaching a sufficiently large value at a sufficiently late time keeps the auxiliary solution above that level thereafter, yielding persistent growth of $u$ on these expanding balls.

\begin{lemma}[Unboundedness of $\zeta$]\label{lem:delay}
  Let $\beta,a,c,t_0>0$, and let
  $F:[0,\infty)\to[0,\infty)$ be continuous and nondecreasing, with $F(0)=0$.
  Assume that
  \begin{align*}
    \int_0^{\delta_0}\frac{F(s)}{s^{2+\frac{1}{\beta}}}\dd s=\infty
    \ \ \text{for some}\ \ \delta_0>0.
  \end{align*}
  We introduce the following delayed integral equation for $\zeta=\zeta(t)$ to construct a lower bound for $u$:
  \begin{align}\label{eq:delay-integral}
    \zeta(t)=
    \begin{cases}
      at^{-\beta}&\text{if}\ \ t_0\leqslant t\leqslant2t_0,\\[2mm]
      t^{-\beta}\displaystyle\left(
      a+c\int_{t_0}^{\frac{t}{2}}s^\beta F\big(\zeta(s)\big)\dd s\right)
      &\text{if} \ \ t\geqslant2t_0.
    \end{cases}
  \end{align}
  This defines a unique positive continuous function, finite on every compact time interval.
  With $\kappa:=\frac{c}{2^{\beta+1}}$, it satisfies
  \begin{align}\label{eq:delay-ode}
    \zeta'(t)=\kappa F\left(\zeta\left(\frac{t}{2}\right)\right)-\frac{\beta}{t}\zeta(t)
    \ \ \text{for}\ \ t>2t_0.
  \end{align}
  Moreover, $\zeta$ is unbounded on $[t_0,\infty)$.
\end{lemma}

\begin{proof}
  If $F(s_*)=0$ for some $s_*>0$, nonnegativity and monotonicity imply $F=0$ on $[0,s_*]$, contradicting the assumed divergence of the Dini-type integral.
  Hence $F(s)>0$ for every $s>0$.

  We construct $\zeta$ successively on $[2^jt_0,2^{j+1}t_0]$, with integers $j\geqslant1$, since the integral in \eqref{eq:delay-integral} only involves values of $\zeta$ already determined on $[t_0,2^jt_0]$.
  Continuity of $F$ and boundedness of $\zeta$ on this compact interval ensure that the integral is finite and defines a continuous extension whose value agrees with $\zeta$ at $2^jt_0$.
  Thus $\zeta$ is defined and continuous on $[t_0,\infty)$, and $a>0$ gives $\zeta(t)>0$.

  Differentiating \eqref{eq:delay-integral} for $t>2t_0$, we obtain \eqref{eq:delay-ode} and
  \begin{align}\label{eq:delay-rescaled}
    z(t):=t^\beta \zeta(t)
    \ \ \text{satisfies}\ \ 
    z'(t)=\kappa t^\beta F\left(\zeta\left(\frac{t}{2}\right)\right)>0
    \ \ \text{for}\ \ t>2t_0.
  \end{align}
  For $t\in[t_0,2t_0]$, we have $z(t)=a$.

  Suppose, for contradiction, that $\zeta(t)\leqslant M$ for all $t\geqslant t_0$.
  We first show that $\zeta(t)\to0$. For $r\in[t,2t]$, monotonicity of $z$ gives
  \begin{align*}
    \zeta(r)\geqslant 2^{-\beta}\zeta(t).
  \end{align*}
  Recall that $\zeta(4t)\leqslant M$.
  Integrating \eqref{eq:delay-rescaled} over $[2t,4t]$, we obtain
  \begin{align}\label{eq:bounded-forces-zero}
    \zeta(4t)
    &\geqslant\kappa(4t)^{-\beta}
    \int_{2t}^{4t}r^\beta F\left(\zeta\left(\frac{r}{2}\right)\right)\dd r\notag\\
    &\geqslant\kappa(4t)^{-\beta}F\big(2^{-\beta}\zeta(t)\big)
    \int_{2t}^{4t}r^\beta\dd r
    =\kappa C_\beta t\,F\big(2^{-\beta}\zeta(t)\big),
  \end{align}
  where
  \begin{align*}
    C_\beta:=\frac{4^{\beta+1}-2^{\beta+1}}{4^\beta(\beta+1)}>0.
  \end{align*}
  Since $\zeta(4t)\leqslant M$, we obtain
  \begin{align*}
    F\big(2^{-\beta}\zeta(t)\big)\leqslant\frac{M}{\kappa C_\beta t}.
  \end{align*}
  Monotonicity of $F$ and $F(s)>0$ for $s>0$ therefore imply $\zeta(t)\to0$ as $t\to\infty$.

  To use this decay without assuming that $\zeta$ is monotone, we introduce its running minimum
  \begin{align}\label{eq:running-min}
    \underline{\zeta}(t):=\min_{t_0\leqslant r\leqslant t}\zeta(r).
  \end{align}
  Since $\zeta$ is locally Lipschitz, its running minimum $\underline{\zeta}$ is locally Lipschitz as well. It is positive on finite intervals, nonincreasing, and tends to zero.
  We prove that
  \begin{align}\label{eq:min-rhs}
    \kappa F\big(\underline{\zeta}(t)\big)-\frac{\beta}{t}\underline{\zeta}(t)\leqslant0
    \ \ \text{for}\ \ t\geqslant2t_0.
  \end{align}
  If this failed at a time $t_*$, put $d=\underline{\zeta}(t_*)>0$.
  All preceding values of $\zeta$ on $[t_0,t_*]$ are at least $d$.
  On $[t_*,2t_*]$, the delayed values are therefore at least $d$.
  By \eqref{eq:delay-ode},
  \begin{align*}
    \frac{\mathrm d}{\mathrm dr}\big[r^\beta(\zeta(r)-d)\big]
    &=r^\beta\left(\kappa F\left(\zeta\left(\frac{r}{2}\right)\right)-\frac{\beta d}{r}\right)\\
    &\geqslant r^\beta\left(\kappa F(d)-\frac{\beta d}{t_*}\right)>0.
  \end{align*}
  Since $\zeta(t_*)\geqslant d$, integration gives $\zeta\geqslant d$ on this interval.
  Repeating on $[2^jt_*,2^{j+1}t_*]$ gives $\zeta(r)\geqslant d$ for all
  $r\geqslant t_*$, contradicting $\zeta(r)\to0$.
  This proves \eqref{eq:min-rhs}.

  For almost every $t>2t_0$, its derivative satisfies
  \begin{align*}
    \underline{\zeta}'(t)=
    \begin{cases}
      0&\mbox{if}\ \ \zeta(t)>\underline{\zeta}(t),\\
      \min\{\zeta'(t),0\}&\mbox{if}\ \ \zeta(t)=\underline{\zeta}(t).
    \end{cases}
  \end{align*}
  If $\zeta(t)=\underline{\zeta}(t)$, then \eqref{eq:delay-ode}, monotonicity of $F$ and $\zeta\big(\frac{t}{2}\big)\geqslant \underline{\zeta}(t)$ give
  \begin{align*}
    \zeta'(t)\geqslant\kappa F\big(\underline{\zeta}(t)\big)-\frac{\beta}{t}\underline{\zeta}(t).
  \end{align*}
  Since the right-hand side is nonpositive by \eqref{eq:min-rhs}, both cases yield
  \begin{align}\label{eq:min-inequality}
    \underline{\zeta}'(t)\geqslant\kappa F\big(\underline{\zeta}(t)\big)-\frac{\beta}{t}\underline{\zeta}(t)
    \ \ \text{for almost every}\ \ t>2t_0.
  \end{align}

  Fix $t_1>2t_0$ and set $Z(t):=t^\beta \underline{\zeta}(t)$.
  By \eqref{eq:min-inequality}, we derive
  \begin{align*}
    Z'(t)\geqslant\kappa t^\beta F\big(\underline{\zeta}(t)\big)>0
    \ \ \text{and}\ \ 
    0\leqslant-\underline{\zeta}'(t)\leqslant\frac{\beta}{t}\underline{\zeta}(t)
  \end{align*}
  almost everywhere.
  For each finite $T>t_1$, the chain rule for the absolutely continuous
  function $\underline{\zeta}$ gives
  \begin{align}\label{eq:critical-change}
    \kappa\int_{\underline{\zeta}(T)}^{\underline{\zeta}(t_1)}\frac{F(s)}{s^{2+\frac{1}{\beta}}}\dd s
    &=\int_{t_1}^T
    \frac{\kappa F\big(\underline{\zeta}(t)\big)}{[\underline{\zeta}(t)]^{2+\frac{1}{\beta}}}\big(-\underline{\zeta}'(t)\big)\dd t
    \leqslant\beta\int_{t_1}^T
    \frac{\kappa F\big(\underline{\zeta}(t)\big)}{t[\underline{\zeta}(t)]^{1+\frac{1}{\beta}}}\dd t\notag\\
    &\leqslant\beta\int_{t_1}^T
    \frac{Z'(t)}{[Z(t)]^{1+\frac{1}{\beta}}}\dd t
    \leqslant\beta^2[Z(t_1)]^{-\frac{1}{\beta}}.
  \end{align}
  Here, the second inequality follows from $Z(t)=t^\beta \underline{\zeta}(t)$ and the preceding lower bound for $Z'(t)$.
  The bound is independent of $T$, whereas the left-hand side diverges as $T\to\infty$ because $\underline{\zeta}(T)\to0$.
  This contradicts the uniform bound and proves the lemma.
\end{proof}

\begin{lemma}[Divergence of $\zeta$ at infinity]\label{lem:persistence}
  Under the assumptions of Lemma~\ref{lem:delay}, the function $\zeta$ defined by \eqref{eq:delay-integral} satisfies
  \begin{align}\label{eq:X-infinity}
    \zeta(t)\to+\infty\ \ \text{as}\ \ t\to\infty.
  \end{align}
\end{lemma}
\begin{proof}
  Fix $d>0$.
  Thanks to $\limsup_{t\to\infty}\zeta(t)=+\infty$ by Lemma~\ref{lem:delay}, we may choose $t_*\geqslant2t_0$ such that
  \begin{align*}
    \zeta(t_*)\geqslant2^\beta d
    \ \ \text{and}\ \ 
    t_*\geqslant\frac{\beta d}{\kappa F(d)}.
  \end{align*}
  Monotonicity of $z(t)=t^\beta \zeta(t)$ gives
  \begin{align*}
    \zeta(r)\geqslant\left(\frac{t_*}{r}\right)^\beta \zeta(t_*)\geqslant d
    \ \ \mbox{for}\ \ t_*\leqslant r\leqslant2t_*.
  \end{align*}
  On $[2t_*,4t_*]$ the delayed values $\zeta(\frac{r}{2})$ are therefore at least $d$.
  By the choice of $t_*$, we have
  \begin{align*}
    \kappa F(d)-\frac{\beta d}{r}
    \geqslant\kappa F(d)-\frac{\beta d}{t_*}
    \geqslant0
    \ \ \text{for}\ \ r\geqslant2t_*.
  \end{align*}
  Combining these bounds with \eqref{eq:delay-ode}, we obtain
  \begin{align*}
    (\zeta-d)'+\frac{\beta}{r}(\zeta-d)\geqslant0
    \ \ \mbox{on}\ \ [2t_*,4t_*],
  \end{align*}
  with nonnegative initial value. Multiplying the preceding differential inequality by $r^\beta$ and integrating from $2t_*$ to $r$, we obtain
  \begin{align*}
    r^\beta\big(\zeta(r)-d\big)
    \geqslant(2t_*)^\beta\big(\zeta(2t_*)-d\big)
    \geqslant0
    \ \ \Rightarrow\ \ \zeta(r)\geqslant d
  \end{align*}
  for $r\in[2t_*,4t_*]$.
  Repeating this argument successively on $[2^jt_*,2^{j+1}t_*]$ for integers $j\geqslant2$, we obtain $\zeta(r)\geqslant d$ for all $r\geqslant t_*$. Since $d$ was arbitrary,
  \eqref{eq:X-infinity} follows.
\end{proof}

\begin{prop}[Lower bound for weak solutions]\label{prop:direct}
  Let $u$ be a weak solution on $[0,T)$ with data satisfying \eqref{Effective-Mass}, and let $a,t_0$ be as in Lemma~\ref{lem:linear}.
  There exists a constant $c=c(n)>0$ such that the function $\zeta$ defined by \eqref{eq:delay-integral} with these constants and $\beta=\frac{n}{2}$ satisfies
  \begin{align}\label{eq:pde-lower}
    u(t,x)\geqslant \zeta(t)
  \end{align}
  for almost every $(t,x)$ with $t_0\leqslant t<T$ and $|x|\leqslant\sqrt t$.
\end{prop}

\begin{proof}
  For $s\leqslant\frac{t}{2}$, $|x|\leqslant\sqrt t$ and $|y|\leqslant\sqrt s$, we have
  \begin{align*}
    |x-y|
    \leqslant\sqrt t+\sqrt s
    \leqslant(\sqrt2+1)\sqrt{t-s}
    <3\sqrt{t-s}.
  \end{align*}
  For $t\geqslant2t_0$ and $s\in\left[t_0,\frac{t}{2}\right]$, we have $t-s\geqslant\frac{t}{2}\geqslant t_0\geqslant18$.
  Lemma~\ref{lem:kernel} with $A=3$ gives
  \begin{align}\label{eq:two-time-kernel}
    k_n(t-s,x-y)
    \geqslant c_{n,3}(t-s)^{-\frac{n}{2}}
    \geqslant c_{n,3}t^{-\frac{n}{2}}.
  \end{align}
  The weak Duhamel formula, nonnegativity of its source, and
  \eqref{eq:linear-seed} therefore imply
  \begin{align}\label{eq:actual-source}
    u(t,x)
    &\geqslant at^{-\frac{n}{2}}
    +c_{n,3}t^{-\frac{n}{2}}
    \int_{t_0}^{\frac{t}{2}}\int_{|y|\leqslant\sqrt s}
    F\big(|u(s,y)|\big)\dd y\dd s.
  \end{align}
  For almost every $(t,x)$ with $t\in[t_0,2t_0]\cap[0,T)$ and $|x|\leqslant\sqrt t$, \eqref{eq:negative-part} and \eqref{eq:linear-seed} give
  \begin{align*}
    u(t,x)\geqslant at^{-\frac{n}{2}}=\zeta(t).
  \end{align*}
  Suppose that \eqref{eq:pde-lower} holds for $t\in[t_0,2^jt_0]\cap[0,T)$, where $j\geqslant1$.
  If $t\in[2^jt_0,2^{j+1}t_0]\cap[0,T)$, then every source time $s\in[t_0,\frac{t}{2}]$ belongs to the preceding interval.
  The induction hypothesis and monotonicity of $F$ therefore yield
  \begin{align*}
    \int_{|y|\leqslant\sqrt s}F\big(|u(s,y)|\big)\dd y
    &\geqslant\int_{|y|\leqslant\sqrt s}F\big(\zeta(s)\big)\dd y
    =\omega_n s^{\frac{n}{2}}F\big(\zeta(s)\big)
  \end{align*}
  for almost every such $s$, where $\omega_n:=|B_1|$.
  Substituting this estimate into \eqref{eq:actual-source}, with Tonelli's theorem justifying the integrations, we obtain
  \begin{align*}
    u(t,x)
    &\geqslant t^{-\frac{n}{2}}
    \left(a+\omega_n c_{n,3}\int_{t_0}^{\frac{t}{2}}
    s^{\frac{n}{2}}F\big(\zeta(s)\big)\dd s\right)
    =\zeta(t)
  \end{align*}
  for almost every $(t,x)$ in the next interval with $|x|\leqslant\sqrt t$, upon choosing $c=\omega_n c_{n,3}$.
  This completes the induction.
\end{proof}

\begin{coro}[Grow-up of global weak solutions]\label{cor:growup}
  Under the near-zero and data assumptions of Theorem~\ref{Thm-Blowup},
  every global weak solution satisfies, for each $A>0$,
  \begin{align*}
    u(t,x)\geqslant A
  \end{align*}
  for almost every $(t,x)$ with $t\geqslant T_A,\ |x|\leqslant\sqrt t$, for some finite $T_A$. For a global mild Sobolev solution this holds at
  every such point and yields \eqref{Global-Growup}.
\end{coro}

\section{The Keller--Osserman criterion}\label{Sec-Blowup}

\subsection{The scalar tail criterion}\label{sec:tail}

To estimate the scalar velocity through the energy identity, we write the primitive in \eqref{Tail-Integral} as
\begin{align*}
  P(s):=\int_0^sF(r)\dd r.
\end{align*}

\begin{lemma}[Consequences of the tail integral]\label{lem:tail}
  Let $F\geqslant0$ be continuous and nondecreasing.
  If $\mathcal K_\infty(F)<\infty$, then
  \begin{align}\label{eq:tail-ratios}
    \frac{P(s)}{s^2}\to\infty,\ \
    \frac{F(s)}{s}\to\infty
    \ \ \text{and}\ \ \frac{\sqrt{P(s)}}{F(s)}\to0
    \ \ \text{as}\ \ s\to\infty.
  \end{align}
  For every such $F\not\equiv0$ and every sufficiently large $R$, we have
  \begin{align*}
    \mathcal K_\infty(F)<\infty
    \ \ \Leftrightarrow\ \ 
    \int_R^\infty\frac{\dd s}{\sqrt{P(s)}}<\infty.
  \end{align*}
\end{lemma}
\begin{proof}
  Assume first that $\mathcal K_\infty(F)<\infty$.
  For $s\geqslant2$, monotonicity of $P$ gives
  \begin{align*}
    \int_{\frac{s}{2}}^s\frac{\dd r}{\sqrt{1+P(r)}}
    \geqslant\frac{s}{2\sqrt{1+P(s)}}.
  \end{align*}
  The left-hand side tends to zero, so $\frac{P(s)}{s^2}\to\infty$.
  Since $P(s)\leqslant sF(s)$, the other two conclusions follow from
  \begin{align*}
    \frac{F(s)}{s}\geqslant\frac{P(s)}{s^2}\ \ \text{and}\ \ 
    \frac{\sqrt{P(s)}}{F(s)}\leqslant\frac{s}{\sqrt{P(s)}}.
  \end{align*}

  To prove the equivalence, assume only that $F\not\equiv0$.
  Nonnegativity and monotonicity imply that $P(s)\to\infty$ as $s\to\infty$.
  Choosing $R$ sufficiently large so that $P(s)\geqslant1$ for $s\geqslant R$, we obtain
  \begin{align*}
    \frac{1}{\sqrt{2P(s)}}
    \leqslant\frac{1}{\sqrt{1+P(s)}}
    \leqslant\frac{1}{\sqrt{P(s)}}
    \ \ \text{for}\ \ s\geqslant R.
  \end{align*}
  Since the integral defining $\mathcal K_\infty(F)$ is finite on $[1,R]$,
  these inequalities prove the asserted equivalence.
\end{proof}

\begin{prop}[Damped scalar equation]\label{prop:scalar}
  Let $F\geqslant0$ be locally Lipschitz and nondecreasing on $[0,\infty)$,
  and let $c>0$, $A>0$ and $B\geqslant0$.
  Let $y$ be the maximal solution, with maximal existence time $\tau_*$, to
  \begin{align}\label{eq:scalar}
    \begin{cases}
      y''+y'=cF(y),&t>0,\\
      (y,y')(0)=(A,B).
    \end{cases}
  \end{align}
  Then the following alternatives hold.
  \begin{enumerate}[label=\textup{(\roman*)}]
    \item If $\mathcal K_\infty(F)<\infty$ and $F(A)>0$, then $\tau_*<\infty$ and
    \begin{align}\label{eq:scalar-source}
      \lim_{t\to\tau_*^-}y(t)=\infty
      \ \ \text{and}\ \ 
      \int_0^{\tau_*}F\big(y(t)\big)\dd t=\infty.
    \end{align}
    \item If $\mathcal K_\infty(F)=\infty$, then $\tau_*=\infty$.
    If, in addition, $F(A)>0$, then
    \begin{align*}
      y(t)\to\infty
      \ \ \text{as}\ \ t\to\infty.
    \end{align*}
  \end{enumerate}
\end{prop}
\begin{proof}
  Multiplying \eqref{eq:scalar} by $\mathrm e^t$ and integrating twice, we obtain
  \begin{align}\label{eq:scalar-volterra}
    y'(t)&=B\mathrm e^{-t}+c\int_0^t \mathrm e^{-(t-s)}F\big(y(s)\big)\dd s,\notag\\
    y(t)&=A+B(1-\mathrm e^{-t})+
    c\int_0^t(1-\mathrm e^{-(t-s)})F\big(y(s)\big)\dd s.
  \end{align}
  They show $y'\geqslant0$ and $y\geqslant A$.
  If $F(A)>0$, the first identity gives
  \begin{align*}
    y'(t)\geqslant cF(A)(1-\mathrm e^{-t}).
  \end{align*}
  Hence, if $\tau_*=\infty$, then $y(t)\to\infty$ as $t\to\infty$.

  At a finite maximal endpoint $\tau_*$, suppose instead that $y\leqslant M<\infty$.
  The same identity yields
  \begin{align*}
    0\leqslant y'(t)\leqslant B+cF(M)
    \ \ \text{for}\ \ 0\leqslant t<\tau_*.
  \end{align*}
  Thus both $y$ and $y'$ remain bounded on $[0,\tau_*)$,
  so the solution extends beyond $\tau_*$, contradicting its maximality.
  Namely, $y$ is unbounded and, by monotonicity, $y(t)\to\infty$ as $t\to\tau_*^-$.

  Multiplying \eqref{eq:scalar} by $y'$ and integrating over $[0,t]$, we obtain
  \begin{align}\label{eq:scalar-energy}
    \frac{1}{2}[y'(t)]^2+\int_0^t[y'(r)]^2\dd r
    =\frac{1}{2}B^2+c\Big(P\big(y(t)\big)-P(A)\Big).
  \end{align}
  Consequently,
  \begin{align*}
    y'(t)\leqslant C\sqrt{1+P\big(y(t)\big)}
    \ \ \text{with}\ \ C:=\sqrt{B^2+2c}.
  \end{align*}
  If $\mathcal K_\infty(F)=\infty$, the preceding estimate and $y(0)=A$ give
  \begin{align*}
    \frac{\mathrm d}{\mathrm dt}
    \left(\int_A^{y(t)}\frac{\dd s}{\sqrt{1+P(s)}}\right)
    =\frac{y'(t)}{\sqrt{1+P\big(y(t)\big)}}
    \leqslant C
    \ \ \Rightarrow\ \
    \int_A^{y(t)}\frac{\dd s}{\sqrt{1+P(s)}}\leqslant Ct.
  \end{align*}
  If $\tau_*<\infty$, then $y(t)\to\infty$ as $t\to\tau_*^-$,
  so the left-hand side tends to infinity by $\mathcal K_\infty(F)=\infty$,
  whereas the right-hand side is bounded by $C\tau_*$.
  This contradiction proves $\tau_*=\infty$.

  Suppose now that $\mathcal K_\infty(F)<\infty$ and $F(A)>0$.
  If $y$ were global, it would tend to infinity. By Lemma~\ref{lem:tail}, we can choose
  $R>A$ with $P(R)>0$ so large that
  \begin{align}\label{eq:damping-absorb}
    \sqrt{B^2+2cP(s)}\leqslant\frac{c}{2}F(s)
    \ \ \mbox{for}\ \ s\geqslant R.
  \end{align}
  Let $t_R$ be the first time such that $y(t_R)=R$.
  By \eqref{eq:scalar-energy} and \eqref{eq:damping-absorb}, we obtain
  \begin{align*}
    y''(t)
    &=cF\big(y(t)\big)-y'(t)\\
    &\geqslant cF\big(y(t)\big)
    -\sqrt{B^2+2cP\big(y(t)\big)}\\
    &\geqslant\frac{c}{2}F\big(y(t)\big)
  \end{align*}
  for $t\geqslant t_R$.
  Multiplying the preceding inequality by $y'(t)\geqslant0$
  and integrating the resultant over $[t_R,t]$, we obtain
  \begin{align*}
    [y'(t)]^2
    \geqslant[y'(t_R)]^2
    +c\Big(P\big(y(t)\big)-P(R)\Big).
  \end{align*}
  Choose $t_1\geqslant t_R$ with $P\big(y(t_1)\big)\geqslant2P(R)$. Then
  \begin{align*}
    y'(t)\geqslant\sqrt{\frac{c}{2}}\sqrt{P\big(y(t)\big)}
    \ \ \mbox{for}\ \ t\geqslant t_1.
  \end{align*}
  Dividing the preceding inequality by $\sqrt{P\big(y(t)\big)}$
  and integrating over $[t_1,t]$, we obtain
  \begin{align*}
    t-t_1
    &\leqslant\sqrt{\frac{2}{c}}
    \int_{t_1}^t\frac{y'(r)}{\sqrt{P\big(y(r)\big)}}\dd r
    =\sqrt{\frac{2}{c}}
    \int_{y(t_1)}^{y(t)}\frac{\dd s}{\sqrt{P(s)}}\\
    &\leqslant\sqrt{\frac{2}{c}}
    \int_{y(t_1)}^\infty\frac{\dd s}{\sqrt{P(s)}}<\infty.
  \end{align*}
  The last integral is finite by Lemma~\ref{lem:tail} and independent of $t$, 
  which contradicts $t\to\infty$. Thus $\tau_*<\infty$ and $y(t)\to\infty$ as $t\to\tau_*^-$.
  Integrating the equation once gives
  \begin{align*}
    y'(t)+y(t)=B+A+c\int_0^tF\big(y(s)\big)\dd s.
  \end{align*}
  The left-hand side tends to infinity at $\tau_*$ because $y'\geqslant0$
  and $y\to\infty$. This proves \eqref{eq:scalar-source}.
\end{proof}

\subsection{Global existence under the divergent Keller--Osserman integral}

\begin{proof}[Proof of Theorem~\ref{Thm-Tail-Global}]
  Let $u$ be the maximal mild Sobolev solution from Lemma~\ref{Lemma-Local}, and define its amplitude by
  \begin{align*}
    U(t):=\|u(t,\cdot)\|_{L^\infty}.
  \end{align*}
  For every $(u_0,u_1)\in H^{N_n}\times H^{N_n-1}$, the linear solution $u^{\mathrm{lin}}$ from Lemma~\ref{lem:linear} satisfies
  \begin{align*}
    \sup_{t\geqslant0}\|u^{\mathrm{lin}}(t,\cdot)\|_{L^\infty}\leqslant L<\infty
  \end{align*}
  by \eqref{eq:uniform-multipliers} and Sobolev embedding.
  Taking the $L^\infty$ norm in the Duhamel formula and using
  the monotonicity of $F$, the nonnegativity of $k_n$ and
  \eqref{eq:kernel-mass}, we obtain, for every $t<T_{\max}$,
  \begin{align}\label{eq:upper-Volterra}
    U(t)\leqslant L+\int_0^t(1-\mathrm e^{-(t-s)})F\big(U(s)\big)\dd s.
  \end{align}

  Let $Y=Y(t)$ solve
  \begin{align*}
    \begin{cases}
      Y''+Y'=F(Y),&t>0,\\
      (Y,Y')(0)=(L+1,0).
    \end{cases}
  \end{align*}
  By Proposition~\ref{prop:scalar}, $Y$ is global in-time,
  including the case $F(L+1)=0$.
  Since $u\in\ml C\big([0,T_{\max}),H^{N_n}\big)$
  and $H^{N_n}\hookrightarrow L^\infty$, the function $U$ is continuous.
  We now prove that
  \begin{align*}
    U(t)\leqslant Y(t)\ \ \mbox{for}\ \ 0\leqslant t<T_{\max}.
  \end{align*}
  Fix $T<T_{\max}$ and choose a Lipschitz constant $C_T$ for $F$ on the compact amplitude interval attained by $U$ and $Y$ on $[0,T]$.
  Subtracting the scalar Volterra identity from \eqref{eq:upper-Volterra}, discarding the resulting constant $-1$, and taking positive parts, we obtain
  \begin{align*}
    \big(U(t)-Y(t)\big)_+
    &\leqslant\int_0^t(1-\mathrm e^{-(t-s)})\big(F\big(U(s)\big)-F\big(Y(s)\big)\big)_+\dd s\\
    &\leqslant C_T\int_0^t\big(U(s)-Y(s)\big)_+\dd s.
  \end{align*}
  Here, the second inequality uses monotonicity and local Lipschitz continuity of $F$, together with $0\leqslant1-\mathrm e^{-(t-s)}\leqslant1$.
  Gronwall's inequality yields $\big(U(t)-Y(t)\big)_+=0$,
  that is, $U(t)\leqslant Y(t)$ for every $t\in[0,T]$.
  Since $T<T_{\max}$ was arbitrary, the comparison holds throughout the maximal interval.
  If $T_{\max}<\infty$, continuity of the global scalar solution gives
  \begin{align*}
    \sup_{0\leqslant t<T_{\max}}\|u(t,\cdot)\|_{L^\infty}
    \leqslant\max_{0\leqslant t\leqslant T_{\max}}Y(t)<\infty,
  \end{align*}
  contradicting \eqref{Continuation-Criterion}.
  Thus $T_{\max}=\infty$.
  Uniqueness follows by applying Lemma~\ref{Lemma-Local}
  successively on finite time intervals.
\end{proof}

\subsection{Lower comparison on a spatial cone}

The lower bound $u(t,x)\geqslant\zeta(t)$ uses only the source integral
over $[t_0,\frac{t}{2}]$ in the Duhamel formula.
To obtain a further lower bound, we also use the source integral
over a time interval disjoint from $[t_0,\frac{t}{2}]$.

\begin{lemma}[Separation of the early and late source terms]\label{lem:reserved}
  Let $u$ be a weak solution under the near-zero and data assumptions of
  Theorem~\ref{Thm-Blowup}. Fix a time $T\geqslant2t_0$. For almost every
  $(t,x)$ in its existence domain with $T\leqslant t\leqslant2T$ and
  $|x|\leqslant\sqrt t$, one has
  \begin{align}\label{eq:reserved}
    u(t,x)\geqslant \zeta(t)+\int_T^t\int_{\mb R^n}
    k_n(t-s,x-y)F\big(|u(s,y)|\big)\dd y\dd s.
  \end{align}
\end{lemma}
\begin{proof}
  Since $\frac{t}{2}\leqslant T$, the source intervals
  $[t_0,\frac{t}{2}]$ and $[T,t]$ have disjoint interiors.
  We retain both intervals in the original Duhamel formula and use nonnegativity of the discarded source terms. For $s\in[t_0,\frac{t}{2}]$, we restrict the spatial integral
  to $\{|y|\leqslant\sqrt s\}$ and apply \eqref{eq:two-time-kernel}
  together with $u(s,y)\geqslant\zeta(s)>0$ on this ball. Combining the lower bound for $u^{\mathrm{lin}}(t,x)$
  with the source integral over $[t_0,\frac{t}{2}]$, we obtain that
  \begin{align*}
    &u^{\mathrm{lin}}(t,x)+\int_{t_0}^{\frac{t}{2}}\int_{B_{\sqrt s}}
    k_n(t-s,x-y)F\big(|u(s,y)|\big)\dd y\dd s\\
    &\geqslant t^{-\frac{n}{2}}\left(a+c\int_{t_0}^{\frac{t}{2}}s^{\frac{n}{2}} F\big(\zeta(s)\big)\dd s\right)=\zeta(t).
  \end{align*}
  Keeping the source integral over $[T,t]\times\mb R^n$
  and combining it with the preceding lower bound, we obtain
  \eqref{eq:reserved}.
\end{proof}

\begin{prop}[Scalar lower comparison on a spatial cone]\label{prop:cone}
  Suppose a global weak solution exists under the near-zero and data
  assumptions of Theorem~\ref{Thm-Blowup}. Fix $A>0$ and let $y=y(r)$ solve
  \begin{align}\label{eq:cone-scalar}
    \begin{cases}
      y''+y'=F(y),&r>0,\\
      (y,y')(0)=(A,0).
    \end{cases}
  \end{align}
  Let $H>0$, and choose $R>H$. There is $T$ such that
  \begin{align}\label{eq:cone-comparison}
    u(T+r,x)\geqslant y(r)
  \end{align}
  for almost every $(r,x)$ with $0\leqslant r<\min\{H,\tau_*\}$
  and $|x|<R-r$, where $\tau_*$ is the maximal existence time of $y$.
\end{prop}
\begin{proof}
  By Lemma~\ref{lem:persistence}, we choose $T$ so large that
  \begin{align}\label{eq:choice-T}
    T\geqslant\max\{2t_0,H,R^2\}
    \ \ \text{and}\ \
    \zeta(t)\geqslant A
    \ \ \text{for}\ \ t\geqslant T.
  \end{align}
  By the choice of $T$, for $0\leqslant r\leqslant H$, we have
  \begin{align*}
    T+r\leqslant2T
    \ \ \text{and}\ \
    R-r\leqslant R\leqslant\sqrt T\leqslant\sqrt{T+r}.
  \end{align*}
  Applying Lemma~\ref{lem:reserved} and using $\zeta(T+r)\geqslant A$,
  we obtain, for almost every $(r,x)$ with $0\leqslant r\leqslant H$
  and $|x|<R-r$,
  \begin{align}\label{eq:cone-Volterra}
    u(T+r,x)\geqslant A+
    \int_0^r\int_{B_{R-q}}
    k_n(r-q,x-y)F\big(|u(T+q,y)|\big)\dd y\dd q.
  \end{align}
  In particular, $u\geqslant A>0$ there.
  For $0\leqslant q\leqslant r$ and $|x|<R-r$, we have
  \begin{align}\label{eq:cone-geometry}
    |x-y|\leqslant r-q
    \ \ \Rightarrow\ \
    |y|\leqslant|x|+|x-y|<R-q.
  \end{align}
  By the support property of $k_n$ and \eqref{eq:kernel-mass}, we obtain
  \begin{align}\label{eq:cone-mass}
    \int_{B_{R-q}}k_n(r-q,x-y)\dd y
    =\int_{\mb R^n}k_n(r-q,x-y)\dd y
    =1-\mathrm e^{-(r-q)}.
  \end{align}

  To compare $u(T+r,x)$ with the scalar solution, we define
  \begin{align*}
    y_0(r):=A
    \ \ \text{and}\ \
    y_{j+1}(r):=A+\int_0^r(1-\mathrm e^{-(r-q)})
    F\big(y_j(q)\big)\dd q
  \end{align*}
  for integers $j\geqslant0$.
  Since $u(T+r,x)\geqslant A=y_0(r)$, the comparison holds for $j=0$.
  Suppose that $u(T+q,y)\geqslant y_j(q)$ almost everywhere
  for $0\leqslant q\leqslant H$ and $|y|<R-q$.
  By \eqref{eq:cone-Volterra}, the monotonicity of $F$ and
  \eqref{eq:cone-mass}, we obtain
  \begin{align*}
    u(T+r,x)
    &\geqslant A+\int_0^r\int_{B_{R-q}}
    k_n(r-q,x-y)F\big(y_j(q)\big)\dd y\dd q\\
    &=A+\int_0^r(1-\mathrm e^{-(r-q)})
    F\big(y_j(q)\big)\dd q\\
    &=y_{j+1}(r).
  \end{align*}
  By induction, $u(T+r,x)\geqslant y_j(r)$ almost everywhere
  for $0\leqslant r\leqslant H$ and $|x|<R-r$,
  for every integer $j\geqslant0$.

  To pass from $u(T+r,x)\geqslant y_j(r)$ to the desired comparison
  with $y(r)$, we prove that $y_j(r)\to y(r)$.
  Fix $0<S<\min\{H,\tau_*\}$.
  Since $1-\mathrm e^{-(r-q)}\geqslant0$ for $0\leqslant q\leqslant r$
  and $F$ is nondecreasing, the integral equations for $y_j$ and $y$ give
  \begin{align*}
    A=y_0(r)\leqslant y_1(r)\leqslant\cdots
    \leqslant  y_{j+1}(r)\leqslant y(r)
    \ \ \text{for}\ \ 0\leqslant r\leqslant S.
  \end{align*}
  Let $\widetilde y(r):=\lim_{j\to\infty}y_j(r)$.
  Using $F\big(y_j(q)\big)\leqslant F\big(y(S)\big)$,
  we pass to the limit by dominated convergence and obtain
  \begin{align*}
    \widetilde y(r)
    =A+\int_0^r(1-\mathrm e^{-(r-q)})
    F\big(\widetilde y(q)\big)\dd q.
  \end{align*}

  Subtracting this identity from that for $y$ and using
  local Lipschitz continuity of $F$ on $[A,y(S)]$, we obtain
  \begin{align*}
    |\widetilde y(r)-y(r)|
    \leqslant C_S\int_0^r|\widetilde y(q)-y(q)|\dd q
    \ \ \text{for}\ \ 0\leqslant r\leqslant S.
  \end{align*}
  Gronwall's inequality gives $\widetilde y=y$.
  Taking a sequence $S_k\uparrow\min\{H,\tau_*\}$ proves \eqref{eq:cone-comparison}.
\end{proof}

\subsection{Finite-time blow-up and global grow-up}

\begin{proof}[Proof of Theorem~\ref{Thm-Blowup}]
  For part \textup{(i)}, suppose that a global weak solution exists and choose any $A>0$.
  The condition \eqref{Intrinsic-NonDini} and monotonicity imply $F(A)>0$.
  By Proposition~\ref{prop:scalar}, the solution $y$ of
  \eqref{eq:cone-scalar} has maximal existence time $\tau_*<\infty$
  and satisfies \eqref{eq:scalar-source}.
  To obtain the lower bound $u(T+r,x)\geqslant y(r)$ on the fixed ball $B_1$,
  we apply Proposition~\ref{prop:cone} with $H=\tau_*$ and $R=\tau_*+1$.
  Since $R-r>1$ for $0\leqslant r<\tau_*$, we obtain
  \begin{align*}
    \int_T^{T+\tau_*}\int_{B_1}F\big(|u(t,x)|\big)\dd x\dd t
    \geqslant |B_1|\int_0^{\tau_*}F\big(y(r)\big)\dd r=\infty.
  \end{align*}
  This contradicts local integrability of the source on the compact set
  $[T,T+\tau_*]\times\overline{B_1}$, which lies inside the domain of the assumed global weak solution.

  Let $u$ now be the maximal mild Sobolev solution from
  Lemma~\ref{Lemma-Local}.
  By Lemma~\ref{lem:mild-weak}, it is also a weak solution,
  so the preceding nonexistence result implies $T_{\max}<\infty$.
  If
  \begin{align*}
    \limsup_{t\to T_{\max}^-}\|u(t,\cdot)\|_{L^\infty}<\infty,
  \end{align*}
  then continuity of $u$ with values in $L^\infty$ on compact time intervals gives
  \begin{align*}
    \sup_{0\leqslant t<T_{\max}}\|u(t,\cdot)\|_{L^\infty}<\infty.
  \end{align*}
  By \eqref{Continuation-Criterion}, this implies $T_{\max}=\infty$,
  a contradiction.
  Thus \eqref{Finite-Blowup} follows.

  For part \textup{(ii)}, Theorem~\ref{Thm-Tail-Global} gives a unique global in-time mild Sobolev solution, and Corollary~\ref{cor:growup} gives \eqref{Global-Growup}.
\end{proof}

\appendix

\section{Verification of the staircase family and its tails}\label{App-Staircase}

\subsection{The staircase family and comparison envelopes}

We prove the properties of $F_a$ stated in Example~\ref{Example-Beyond}, including the integral conclusions for its upper and lower comparison functions.
Throughout this subsection, write $p:=p_{\mathrm F}(n)$ and $\sigma:=\frac{1}{4}$.
By \eqref{Staircase-Cycles}, the lengths of $[a_j,b_j]$ and $[b_j,a_{j+1}]$ are
\begin{align*}
  L_j:=b_j-a_j=j^2
  \ \ \text{and}\ \
  R_j:=a_{j+1}-b_j=(j+1)^2+\frac{a}{p}\log\frac{j+1}{j}.
\end{align*}

To control the lower comparison functions in Example~\ref{Example-Beyond}, for a positive continuous function $F$ on $(0,\infty)$, $b\geqslant1$ and $\delta>0$, we introduce the envelope
\begin{align}\label{Envelope-Definition}
  E_{b,\delta}[F](s):=s^b\min_{s\leqslant r\leqslant\delta}
  \frac{F(r)}{r^b}
  \ \ \mbox{for}\ \ 0<s\leqslant\delta.
\end{align}

\begin{prop}[Integral thresholds and comparison envelopes]\label{Prop-Obstruction}
  For each $a>0$, the nonlinearity $F_a$ in \eqref{Staircase-Nonlinearity} has the following properties.
  \begin{enumerate}[label=\textup{(\roman*)}]
    \item The function $F_a$ satisfies Assumption~\ref{Assum-Chen}, and its quotient $\frac{F_a(s)}{s^p}$ tends to zero as $s\to0^+$ and is nonmonotone on every neighborhood of zero.
    For every $\delta>0$, we have
    \begin{align}\label{Obstruction-Integrals}
      \ml I_{F_a}(\delta)<\infty
      \ \ \Leftrightarrow\ \ a>1.
    \end{align}
    Moreover,
    \begin{align*}
      F_a(s)=\mathrm e^{-p}s^{2p}
      \ \ \mbox{for}\ \ s\geqslant\mathrm e^\sigma.
    \end{align*}

    \item If $1<a\leqslant3$ and $\nu\geqslant0$ is nondecreasing with $F_a(s)\leqslant Cs^p\nu(s)$ near zero for some $C>0$, then
    \begin{align}\label{Monotone-Majorant-Loss}
      \int_0^\delta\frac{\nu(s)}{s}\dd s=\infty
    \end{align}
    for every sufficiently small $\delta>0$.

    \item For every $b\geqslant1$ and $\delta>0$, we have
    \begin{align}\label{Envelope-Critical-Integral}
      \int_0^\delta\frac{E_{b,\delta}[F_a](s)}{s^{p+1}}\dd s<\infty.
    \end{align}
  \end{enumerate}
\end{prop}

\begin{proof}
  \noindent\textbf{Step 1. Properties of $F_a$ and its critical integral.}
  The choice of $R_j$ gives
  \begin{align*}
    j^{-a}\mathrm e^{-pb_j-2pR_j}
    =(j+1)^{-a}\mathrm e^{-p(a_{j+1}+L_{j+1})},
  \end{align*}
  so the values of $f_{0,a}$ agree at the endpoints of consecutive cycles.
  The two pieces meeting at $a_1=0$ both have value $\mathrm e^{-p}$.
  Thus $f_{0,a}$ is continuous and nonincreasing, and its logarithm $\ell_{0,a}:=\log f_{0,a}$ is globally Lipschitz with slopes $0$ and $-2p$.

  Let $F_{0,a}(s):=f_{0,a}(\log\frac{1}{s})$ for $s>0$.
  Under the change of variables $s=\mathrm e^{-t}$, its critical quotient on $[a_j,a_{j+1}]$ satisfies
  \begin{align}\label{eq:barrier-mu}
    \frac{F_{0,a}(\mathrm e^{-t})}{(\mathrm e^{-t})^p}
    =\mathrm e^{pt}f_{0,a}(t)=
    \begin{cases}
      j^{-a}\mathrm e^{-p(b_j-t)}&\mbox{if}\ \ a_j\leqslant t\leqslant b_j,\\
      j^{-a}\mathrm e^{-p(t-b_j)}&\mbox{if}\ \ b_j\leqslant t\leqslant a_{j+1}.
    \end{cases}
  \end{align}
  This quotient is bounded by $j^{-a}$ and tends to zero as $t\to\infty$.
  Integrating over $[a_j,b_j]$ and $[b_j,a_{j+1}]$, we obtain
  \begin{align}\label{eq:barrier-divergence}
    \int_{a_j}^{a_{j+1}}\mathrm e^{pt}f_{0,a}(t)\dd t
    =\frac{j^{-a}}{p}\big(2-\mathrm e^{-pL_j}-\mathrm e^{-pR_j}\big)
    \approx j^{-a},
  \end{align}
  where the comparison constants are independent of $j$, since $L_j\geqslant1$ and $R_j\geqslant4$.
  The change of variables $s=\mathrm e^{-t}$ therefore yields
  \begin{align*}
    \ml I_{F_{0,a}}(1)
    =\int_0^\infty\mathrm e^{pt}f_{0,a}(t)\dd t
    \approx\sum_{j=1}^\infty j^{-a}.
  \end{align*}

  To transfer this comparison to $F_a$, the Lipschitz bound for $\ell_{0,a}$ and \eqref{Staircase-Smoothing} give
  \begin{align*}
    |\ell_a(t)-\ell_{0,a}(t)|
    &\leqslant\int_{-\sigma}^{\sigma}\rho(\tau)
    |\ell_{0,a}(t-\tau)-\ell_{0,a}(t)|\dd\tau
    \leqslant2p\sigma.
  \end{align*}
  Exponentiation at $t=\log\frac{1}{s}$ gives
  \begin{align}\label{eq:smoothing-comparison}
    \mathrm e^{-2p\sigma}F_{0,a}(s)\leqslant F_a(s)
    \leqslant\mathrm e^{2p\sigma}F_{0,a}(s)
    \ \ \mbox{for}\ \ s>0.
  \end{align}
  Hence $\ml I_{F_a}(1)\approx\sum_{j=1}^\infty j^{-a}$, which proves \eqref{Obstruction-Integrals} for $\delta=1$.
  The same criterion holds for every $\delta>0$, since $F_a$ is continuous away from zero.
  By \eqref{eq:barrier-mu} and \eqref{eq:smoothing-comparison}, we also have $\frac{F_a(s)}{s^p}\to0$ as $s\to0^+$.

  Differentiating the convolution in \eqref{Staircase-Smoothing}, we obtain
  \begin{align*}
    -2p\leqslant\ell_a'(t)\leqslant0
    \ \ \mbox{and}\ \ |\ell_a''(t)|\leqslant2p\|\rho'\|_{L^1}.
  \end{align*}
  Thus $F_a$ is smooth and nondecreasing on $(0,\infty)$.
  For $t=\log\frac{1}{s}$, differentiation gives
  \begin{align*}
    F_a'(s)=-\frac{F_a(s)}{s}\ell_a'(t)
    \ \ \text{and}\ \ F_a''(s)=\frac{F_a(s)}{s^2}
    \big([\ell_a'(t)]^2+\ell_a'(t)+\ell_a''(t)\big).
  \end{align*}
  For $n=1$, we have $p=3$ and
  \begin{align*}
    |F_a'(s)|\leqslant2p\frac{F_a(s)}{s}
    =2ps^2\frac{F_a(s)}{s^3},
  \end{align*}
  which tends to zero as $s\to0^+$.
  For $n=2$, the derivative formulas and $p=2$ give
  \begin{align*}
    \frac{|F_a'(s)|}{s}+|F_a''(s)|
    \leqslant C\frac{F_a(s)}{s^2},
  \end{align*}
  which also tends to zero.
  Together with $F_a(0)=0$ and $F_a(s)=o(s^p)$ as $s\to0^+$, these estimates give the $\ml C^1$ even extension with first derivative zero at the origin for $n=1$, and the $\ml C^2$ even extension with first and second derivatives zero at the origin for $n=2$.
  Thus $F_a$ satisfies Assumption~\ref{Assum-Chen}.

  For $a_j+\sigma\leqslant t\leqslant b_j-\sigma$ and $|\tau|\leqslant\sigma$, we have $t-\tau\in[a_j,b_j]$.
  Since $\ell_{0,a}$ is constant on this interval, \eqref{Staircase-Smoothing} gives
  \begin{align*}
    F_a(\mathrm e^{-t})=j^{-a}\mathrm e^{-pb_j}
    \ \ \mbox{for}\ \ a_j+\sigma\leqslant t\leqslant b_j-\sigma.
  \end{align*}
  Taking
  \begin{align*}
    r_j:=\mathrm e^{-(b_j-\sigma)}
    \ \ \mbox{and}\ \ s_j:=\mathrm e^{-(a_j+\sigma)},
  \end{align*}
  we obtain
  \begin{align*}
    \frac{F_a(r_j)r_j^{-p}}{F_a(s_j)s_j^{-p}}
    =\left(\frac{s_j}{r_j}\right)^p
    =\mathrm e^{p(L_j-2\sigma)}
    \to\infty
    \ \ \mbox{as}\ \ j\to\infty.
  \end{align*}
  Since $r_j<s_j\to0$ and $\frac{F_a(r_j)}{r_j^p}>\frac{F_a(s_j)}{s_j^p}$, the quotient cannot be nondecreasing on any neighborhood of zero.
  A nonincreasing quotient is also impossible, since $\frac{F_a(s)}{s^p}>0$ and $\frac{F_a(s)}{s^p}\to0$ as $s\to0^+$.

  For $t\leqslant-\sigma$ and $|\tau|\leqslant\sigma$, we have $t-\tau\leqslant0$, where $\ell_{0,a}(t-\tau)=-p-2p(t-\tau)$.
  Evenness of $\rho$ gives $\int_{-\sigma}^{\sigma}\tau\rho(\tau)\dd\tau=0$, hence $\ell_a(t)=-p-2pt$ on this interval.
  Returning to $s=\mathrm e^{-t}$, we obtain $F_a(s)=\mathrm e^{-p}s^{2p}$ for $s\geqslant\mathrm e^\sigma$.
  This proves part \textup{(i)}.

  \medskip
  \noindent\textbf{Step 2. Nondecreasing upper bounds for $\frac{F_a(s)}{s^p}$.}
  Let $1<a\leqslant3$ and let $\nu$ satisfy the assumption in part \textup{(ii)}.
  At $\eta_j:=\mathrm e^{-b_j}$, \eqref{eq:barrier-mu} and \eqref{eq:smoothing-comparison} yield
  \begin{align*}
    \nu(\eta_j)\geqslant\frac{F_a(\eta_j)}{C\eta_j^p}
    \geqslant\frac{\mathrm e^{-2p\sigma}}{C}j^{-a}
  \end{align*}
  for all sufficiently large $j$.
  Since $\nu$ is nondecreasing, it follows that
  \begin{align*}
    \int_{\mathrm e^{-b_j}}^{\mathrm e^{-a_j}}\frac{\nu(s)}{s}\dd s
    &\geqslant\nu(\eta_j)\int_{\mathrm e^{-b_j}}^{\mathrm e^{-a_j}}\frac{\dd s}{s}
    \geqslant c j^{-a}(b_j-a_j)
    =c j^{2-a},
  \end{align*}
  where $c:=\frac{\mathrm e^{-2p\sigma}}{C}>0$ is independent of $j$.

  Fix a sufficiently small $\delta>0$ and choose $j_0$ so large that the preceding estimates hold for every $j\geqslant j_0$ and $\mathrm e^{-a_{j_0}}<\delta$.
  Summing over the disjoint intervals $[\mathrm e^{-b_j},\mathrm e^{-a_j}]\subset(0,\delta)$, we obtain
  \begin{align*}
    \int_0^\delta\frac{\nu(s)}{s}\dd s
    &\geqslant\sum_{j=j_0}^\infty
    \int_{\mathrm e^{-b_j}}^{\mathrm e^{-a_j}}\frac{\nu(s)}{s}\dd s
    \geqslant c\sum_{j=j_0}^\infty j^{2-a}
    =\infty,
  \end{align*}
  which proves \eqref{Monotone-Majorant-Loss}.

  \medskip
  \noindent\textbf{Step 3. Integrability of the lower envelopes.}
  Under the change of variables $s=\mathrm e^{-t}$, the integrand in the critical integral of $E_{1,1}[F_{0,a}]$ becomes
  \begin{align*}
    W_a(t):=\mathrm e^{pt}E_{1,1}[F_{0,a}](\mathrm e^{-t})
    =\mathrm e^{(p-1)t}\min_{0\leqslant v\leqslant t}
    \mathrm e^vf_{0,a}(v)
    \ \ \mbox{for}\ \ t\geqslant0.
  \end{align*}
  For $t=a_j+w$ with $0\leqslant w\leqslant L_j$, choosing $v=a_j$ in the minimum gives
  \begin{align*}
    W_a(a_j+w)\leqslant j^{-a}\mathrm e^{-pL_j}\mathrm e^{(p-1)w}.
  \end{align*}
  Integrating over $[0,L_j]$, we obtain
  \begin{align*}
    \int_0^{L_j}W_a(a_j+w)\dd w
    \leqslant j^{-a}\mathrm e^{-pL_j}\int_0^{L_j}\mathrm e^{(p-1)w}\dd w
    \leqslant\frac{j^{-a}}{p-1}\mathrm e^{-L_j}.
  \end{align*}
  For $t=b_j+w$ with $0\leqslant w\leqslant R_j$, choosing $v=a_j$ and $v=t$ gives
  \begin{align*}
    W_a(b_j+w)\leqslant j^{-a}
    \min\big\{\mathrm e^{-L_j}\mathrm e^{(p-1)w},\mathrm e^{-pw}\big\}.
  \end{align*}
  The two bounds agree at $w_*:=\frac{L_j}{2p-1}\in(0,R_j)$.
  Splitting the integral at $w_*$ and enlarging the second interval, we derive
  \begin{align*}
    \int_0^{R_j}W_a(b_j+w)\dd w
    &\leqslant j^{-a}\left(
    \mathrm e^{-L_j}\int_0^{w_*}\mathrm e^{(p-1)w}\dd w
    +\int_{w_*}^\infty\mathrm e^{-pw}\dd w\right)\\
    &\leqslant j^{-a}\left(\frac{1}{p-1}+\frac{1}{p}\right)
    \mathrm e^{-\frac{pL_j}{2p-1}}.
  \end{align*}
  Summing over $[a_j,a_{j+1}]$ and using $L_j=j^2$, we obtain
  \begin{align}\label{eq:barrier-envelope-finite}
    \int_0^1\frac{E_{1,1}[F_{0,a}](s)}{s^{p+1}}\dd s
    &=\int_0^\infty W_a(t)\dd t\notag\\
    &\leqslant C\sum_{j=1}^\infty j^{-a}
    \left(\mathrm e^{-j^2}
    +\mathrm e^{-\frac{p}{2p-1}j^2}\right)<\infty.
  \end{align}
  By \eqref{eq:smoothing-comparison}, we also have
  \begin{align*}
    E_{1,1}[F_a](s)\leqslant
    \mathrm e^{2p\sigma}E_{1,1}[F_{0,a}](s),
  \end{align*}
  which proves \eqref{Envelope-Critical-Integral} for $b=1$ and $\delta=1$.

  For $b\geqslant1$ and $0<s\leqslant\delta$, we obtain
  \begin{align*}
    E_{b,\delta}[F_a](s)
    &=\min_{s\leqslant r\leqslant\delta}\left(\frac{s}{r}\right)^bF_a(r)
    \leqslant\min_{s\leqslant r\leqslant\delta}\frac{s}{r}F_a(r)
    =E_{1,\delta}[F_a](s).
  \end{align*}
  It therefore remains to treat arbitrary $\delta>0$ for $b=1$.
  If $\delta\geqslant1$, enlarging the interval in the minimum gives
  \begin{align*}
    E_{1,\delta}[F_a](s)\leqslant E_{1,1}[F_a](s)
    \ \ \mbox{for}\ \ 0<s\leqslant1.
  \end{align*}
  If $0<\delta<1$, put
  \begin{align*}
    c_\delta:=\min_{\delta\leqslant r\leqslant1}\frac{F_a(r)}{r}>0.
  \end{align*}
  Since $\frac{F_a(s)}{s}\to0$ as $s\to0^+$, we may choose $s_\delta\in(0,\delta)$ such that $\frac{F_a(s)}{s}<c_\delta$ for $0<s\leqslant s_\delta$.
  The value at $r=s$ is then smaller than every value on $[\delta,1]$, so
  \begin{align*}
    E_{1,\delta}[F_a](s)=E_{1,1}[F_a](s)
    \ \ \mbox{for}\ \ 0<s\leqslant s_\delta.
  \end{align*}
  The remaining integrals away from zero are finite because $0\leqslant E_{1,\delta}[F_a]\leqslant F_a$.
  This proves \eqref{Envelope-Critical-Integral} for every $b\geqslant1$ and $\delta>0$.
\end{proof}

\begin{remark}
  The estimate \eqref{Envelope-Critical-Integral} gives the two lower-comparison conclusions in Example~\ref{Example-Beyond}.
  If $G\geqslant0$ is convex on $[0,\delta]$, $G(0)=0$ and $G\leqslant F_a$, then
  \begin{align*}
    G(s)\leqslant\frac{s}{r}G(r)\leqslant\frac{s}{r}F_a(r)
    \ \ \mbox{for}\ \ 0<s\leqslant r\leqslant\delta.
  \end{align*}
  Taking the minimum over $r\in[s,\delta]$ and using \eqref{Envelope-Critical-Integral} with $b=1$, we obtain
  \begin{align*}
    \ml I_G(\delta)
    \leqslant\int_0^\delta\frac{E_{1,\delta}[F_a](s)}{s^{p+1}}\dd s<\infty.
  \end{align*}
  If $\nu\geqslant0$ is nondecreasing and $s^p\nu(s)\leqslant CF_a(s)$ on $(0,\delta]$, then
  \begin{align*}
    \nu(s)\leqslant\nu(r)\leqslant C\frac{F_a(r)}{r^p}
    \ \ \mbox{for}\ \ 0<s\leqslant r\leqslant\delta.
  \end{align*}
  Taking the minimum over $r\in[s,\delta]$ gives $s^p\nu(s)\leqslant CE_{p,\delta}[F_a](s)$, and \eqref{Envelope-Critical-Integral} with $b=p$ yields
  \begin{align*}
    \int_0^\delta\frac{\nu(s)}{s}\dd s
    \leqslant C\int_0^\delta\frac{E_{p,\delta}[F_a](s)}{s^{p+1}}\dd s<\infty.
  \end{align*}
  Both conclusions hold for every $a>0$.
  In particular, when $0<a\leqslant1$, neither type of lower comparison preserves the divergence of $\ml I_{F_a}(\delta)$.
\end{remark}

\subsection{Smooth joining and logarithmic tails}\label{App-Tail-Join}

Fix $a>0$.
To construct the nonlinearities in Remark~\ref{Rem-Tails} without changing $F_a$ near zero, we join its derivative to that of a prescribed tail.
Let $\chi\in\ml C^\infty\big([0,\infty)\big)$ be nondecreasing, with $\chi=0$ on $[0,2]$ and $\chi=1$ on $[3,\infty)$.
For a smooth nondecreasing function $H$ on $(1,\infty)$, define
\begin{align}\label{eq:join}
  F_{a,H}(s):=\int_0^sD_{a,H}(r)\dd r,
\end{align}
where
\begin{align*}
  D_{a,H}(r):=
  \begin{cases}
    F_a'(r)&\mbox{if}\ \ 0\leqslant r\leqslant2,\\
    (1-\chi(r))F_a'(r)+\chi(r)H'(r)&\mbox{if}\ \ r>2.
  \end{cases}
\end{align*}
The smoothness of $\chi$, together with $\chi=0$ on $[0,2]$ and $\chi=1$ on $[3,\infty)$, ensures that $D_{a,H}$ is smooth on $(0,\infty)$.
Moreover, $F_a'\geqslant0$ and $H'\geqslant0$ give $D_{a,H}\geqslant0$, so $F_{a,H}$ is nonnegative and nondecreasing.
As $F_{a,H}=F_a$ on $[0,2]$, it satisfies Assumption~\ref{Assum-Chen} and retains the critical-integral criterion and both comparison properties in Example~\ref{Example-Beyond}.
For $s\geqslant3$, integration of $D_{a,H}=H'$ gives
\begin{align}\label{eq:tail-shift}
  F_{a,H}(s)=H(s)+C_{a,H}
  \ \ \mbox{with}\ \ C_{a,H}:=F_{a,H}(3)-H(3).
\end{align}

For $H(s)=s(\log s)^\gamma$ with $\gamma\geqslant0$ and $s>1$, we have
\begin{align*}
  H'(s)=(\log s)^\gamma+\gamma(\log s)^{\gamma-1}>0.
\end{align*}
Taking $F_{a,\gamma}:=F_{a,H}$ in \eqref{eq:join}, we obtain $F_{a,\gamma}(s)\sim s(\log s)^\gamma$ as $s\to\infty$ by \eqref{eq:tail-shift}.
To compute its primitive, we use
\begin{align*}
  \frac{\mathrm d}{\mathrm ds}\left[\frac{1}{2}s^2(\log s)^\gamma\right]
  =s(\log s)^\gamma\left(1+\frac{\gamma}{2\log s}\right)
  \sim F_{a,\gamma}(s),
\end{align*}
which, after integration from a fixed sufficiently large lower endpoint to $s$, yields
\begin{align*}
  \int_0^sF_{a,\gamma}(r)\dd r\sim\frac{1}{2}s^2(\log s)^\gamma
  \ \ \mbox{as}\ \ s\to\infty.
\end{align*}
Consequently,
\begin{align*}
  \mathcal K_\infty(F_{a,\gamma})<\infty
  \ \ \Leftrightarrow\ \ 
  \int_3^\infty\frac{\dd s}{s(\log s)^{\frac{\gamma}{2}}}<\infty
  \ \ \Leftrightarrow\ \ \gamma>2.
\end{align*}
This proves \eqref{Tail-Logarithmic-Threshold}.

\subsection{Oscillatory tails}\label{App-Oscillatory-Tails}

The Keller--Osserman threshold $\gamma=2$ in Appendix~\ref{App-Tail-Join} can also hold for tails that are not regularly varying and whose second derivatives change sign arbitrarily far out.
To construct such examples using \eqref{eq:join}, fix $a=1$, write $F_H:=F_{1,H}$, and take
\begin{align*}
  H_\gamma(s):=s(\log s)^\gamma\left(2+\frac{1}{4}\sin(\log s)\right)
  \ \ \mbox{with}\ \ \gamma\geqslant0\ \ \mbox{and}\ \ s>1.
\end{align*}
Differentiation gives
\begin{align*}
  H_\gamma'(s)
  &=(\log s)^\gamma\left(
  2+\frac{1}{4}\sin(\log s)+\frac{1}{4}\cos(\log s)
  +\frac{\gamma}{\log s}\left(2+\frac{1}{4}\sin(\log s)\right)\right)\\
  &\geqslant\left(2-\frac{\sqrt2}{4}\right)(\log s)^\gamma>0.
\end{align*}
Thus the construction in Appendix~\ref{App-Tail-Join} applies to $H_\gamma$.
The bounds
\begin{align*}
  \frac{7}{4}s(\log s)^\gamma\leqslant H_\gamma(s)
  \leqslant\frac{9}{4}s(\log s)^\gamma
\end{align*}
and \eqref{eq:tail-shift} give
\begin{align*}
  \int_0^sF_{H_\gamma}(r)\dd r\approx s^2(\log s)^\gamma
  \ \ \mbox{as}\ \ s\to\infty,
\end{align*}
since the additive constant in \eqref{eq:tail-shift} contributes only a term of order $s$ to the primitive.
The integral comparison in Appendix~\ref{App-Tail-Join} therefore yields
\begin{align*}
  \mathcal K_\infty(F_{H_\gamma})<\infty
  \ \ \Leftrightarrow\ \ \gamma>2.
\end{align*}

To exclude regular variation, choose $\lambda:=\mathrm e^\pi$ and
\begin{align*}
  s_j:=\mathrm e^{2\pi j+\frac{\pi}{2}}
  \ \ \mbox{and}\ \ r_j:=\mathrm e^{2\pi j+\frac{3\pi}{2}}.
\end{align*}
By \eqref{eq:tail-shift}, we obtain
\begin{align*}
  \frac{F_{H_\gamma}(\lambda s_j)}{F_{H_\gamma}(s_j)}\to\frac{7}{9}\lambda
  \ \ \text{and}\ \ 
  \frac{F_{H_\gamma}(\lambda r_j)}{F_{H_\gamma}(r_j)}\to\frac{9}{7}\lambda
  \ \ \mbox{as}\ \ j\to\infty.
\end{align*}
These distinct limits show that $\lim_{s\to\infty}\frac{F_{H_\gamma}(\lambda s)}{F_{H_\gamma}(s)}$ does not exist, so $F_{H_\gamma}$ is not regularly varying.

For the second derivative, \eqref{eq:tail-shift} gives $F_{H_\gamma}''(s)=H_\gamma''(s)$ for $s\geqslant3$, and, for each fixed $\gamma\geqslant0$,
\begin{align*}
  H_\gamma''(s)=\frac{(\log s)^\gamma}{s}
  \left(\frac{1}{4}\cos(\log s)-\frac{1}{4}\sin(\log s)
  +O\big((\log s)^{-1}\big)\right)
  \ \ \mbox{as}\ \ s\to\infty.
\end{align*}
Hence $F_{H_\gamma}''$ changes sign arbitrarily far out.

Since $F_{H_\gamma}=F_1$ near zero, \eqref{Obstruction-Integrals} gives $\ml I_{F_{H_\gamma}}(\delta)=\infty$ for every $\delta>0$.
Theorem~\ref{Thm-Blowup} therefore gives global grow-up if $0\leqslant\gamma\leqslant2$ and finite-time blow-up if $\gamma>2$ for data in $\ml A_n$ with positive effective mass \eqref{Effective-Mass}.

\enlargethispage{2\baselineskip}

\section*{Acknowledgments}
Wenhui Chen is supported in part by the National Natural Science Foundation of China (grant No. 12301270) and the Guangdong Basic and Applied Basic Research Foundation (grant No. 2025A1515010240).  The author thanks Giovanni Girardi for suggestions on the Introduction.

\end{document}